\documentclass[a4paper,10pt]{amsart}

\usepackage[english]{babel}
\usepackage[utf8]{inputenc}
\usepackage{amssymb}
\usepackage{amsmath}
\usepackage{amsthm}
\usepackage{bbm}

\newcommand{\bbC}{\mathbb{C}}
\newcommand{\bbF}{\mathbb{F}}
\newcommand{\bbN}{\mathbb{N}}
\newcommand{\bbQ}{\mathbb{Q}}
\newcommand{\bbR}{\mathbb{R}}
\newcommand{\bbT}{\mathbb{T}}
\newcommand{\bbZ}{\mathbb{Z}}

\newcommand{\calA}{\mathcal{A}}
\newcommand{\calB}{\mathcal{B}}

\newcommand{\calL}{\mathcal{L}}

\newcommand{\calU}{\mathcal{U}}

\newcommand{\re}[1]{\operatorname{Re} #1}
\newcommand{\im}[1]{\operatorname{Im} #1}
\newcommand{\id}{\operatorname{id}}

\theoremstyle{definition}
\newtheorem{definition}{Definition}[section]
\newtheorem{remark}[definition]{Remark}
\newtheorem{remarks}[definition]{Remarks}
\newtheorem{example}[definition]{Example}

\theoremstyle{plain}
\newtheorem{proposition}[definition]{Proposition}	
\newtheorem{lemma}[definition]{Lemma}
\newtheorem{theorem}[definition]{Theorem}	
\newtheorem{corollary}[definition]{Corollary}

\begin{document}

\title[Spectral Properties of Contractive Semigroups]{Spectral and Asymptotic Properties of Contractive Semigroups on Non-Hilbert Spaces}

\author{Jochen Gl\"uck}
\email{jochen.glueck@uni-ulm.de}
\address{Jochen Gl\"uck, Institute of Applied Analysis, Ulm University, 89069 Ulm, Germany}


\date{\today}

\begin{abstract}
	We analyse $C_0$-semigroups of contractive operators on real-valued $L^p$-spaces for $p \not= 2$ and on other classes of non-Hilbert spaces. We show that, under some regularity assumptions on the semigroup, the geometry of the unit ball of those spaces forces the semigroup's generator to have only trivial (point) spectrum on the imaginary axis. This has interesting consequences for the asymptotic behaviour as $t \to \infty$. For example, we can show that a contractive and eventually norm continuous $C_0$-semigroup on a real-valued $L^p$-space automatically converges strongly if $p \not\in \{1,2,\infty\}$. 
\end{abstract}

\maketitle

\section{Introduction and Preliminaries}

If we want to analyse the asymptotic behaviour of a bounded $C_0$-semigroup $(e^{tA})_{t \ge 0}$ on a Banach space $X$, then an important question to ask is whether the spectrum of $A$ on the imaginary axis is trivial, i.e.\ contained in the set $\{0\}$.  If this is the case, then we can often conclude that $e^{tA}$ converges strongly as $t \to \infty$ (see e.g.\ \cite[Corollary 2.6]{Arendt1988}). Therefore, one is interested in simple criteria to ensure that $\sigma(A) \cap i\bbR \subset \{0\}$. A typical example for such a criterion is positivity of the semigroup with respect to a Banach lattice cone (in combination with an additional regularity assumption on the semigroup, see \cite[Corollary C-III.2.13]{Arendt1986}). \par 
In this article, we consider another geometric condition on the semigroup, namely contractivity of each operator $e^{tA}$. We will see that on many important function spaces, this condition ensures that $\sigma(A) \cap i \bbR \subset \{0\}$. To give the reader a more concrete idea of what we are going to do, let us state the following two results which will be proved (in a more general form) in the subsequent sections:

\begin{theorem}
	Let $(\Omega,\Sigma,\mu)$ be a measure space and let $1 < p < \infty$, $p \not=2$. Let $(e^{tA})_{t \ge 0}$ be an eventually norm continuous, contractive $C_0$-semigroup on the real-valued function space $L^p(\Omega,\Sigma,\mu; \bbR)$. Then $\sigma(A) \cap i \bbR \subset \{0\}$. In particular, $e^{tA}$ converges strongly as $t \to \infty$.
\end{theorem}

In this theorem, $\sigma(A)$ denotes the spectrum of the complex extension of $A$ to a complexification of the real space $L^p(\Omega,\Sigma,\mu; \bbR)$, see the end of the introduction for details. The above theorem (in fact, a generalization of it) is proved in Subsection \ref{subsection_contraction_semigroups_on_capital_l_p_spaces} below (Corollaries \ref{cor_trivial_peripheral_spectrum_on_capital_l_p_spaces} and \ref{cor_convergence_of_semigroups_on_l_p_spaces}). 

In the following theorem, $c(\bbN;\bbR)$ denotes the space of all real-valued, convergent sequences (indexed by $\bbN)$ and $c_0(\bbN;\bbR)$ denotes its subspace of sequences which converge to $0$.

\begin{theorem}
	Let $X$ be one of the real-valued sequence spaces $c(\bbN;\bbR)$, $c_0(\bbN;\bbR)$ or $l^p(\bbN;\bbR)$ for $1 \le p < \infty$, $p \not= 2$. If $(e^{tA})_{t\ge 0}$ is a contractive $C_0$-semigroup on $X$, then $\sigma_{\operatorname{pnt}}(A) \cap i \bbR \subset \{0\}$. If, in addition, $A$ has compact resolvent, then $e^{tA}$ converges strongly as $t \to \infty$.
\end{theorem}

Here, $\sigma_{\operatorname{pnt}}(A)$ denotes the point spectrum of a complex extension of $A$ to a complexification of the real space $X$ and we say that $A$ has compact resolvent if the complex extension of $A$ to any complexification of $X$ has compact resolvent; note that this property does not depend on the choice of the complexification (see again the end of the introduction for details). The above theorem follows from results that we are going to prove in Subsection \ref{subsection_c_0_semigroups_on_extremely_non_hilbert_spaces} (see Corollary \ref{cor_peripheral_point_spectrum_and_convergence_on_l_p_sequence_spaces} for the case $X = l^p$; for the cases $X = c(\bbN;\bbR)$ and $X = c_0(\bbN;\bbR)$, see Theorem \ref{theorem_point_spectrum_for_sg_on_extremely_non_hilbert_spaces}, Corollary \ref{cor_convergence_for_compact_resolvent_on_extremely_non_hilbert_space} and the assertions before and after Example \ref{examp_c_k_is_extremely_non_hilbert_for_countable_k}). \par 

Both of the above theorems make use of the fact that the underlying Banach space behaves in some manner oppositely to a Hilbert space (which explains the condition $p \not= 2$ in the assumptions). The idea to employ this observation for the spectral analysis of contractions is not new. A similar approach was used by Krasnosel'ski{\u\i} in \cite{Krasnoselskii1968} and by Lyubich in \cite{Lyubich1970} to analyse the peripheral spectrum of compact and of finite dimensional contractions. Furthermore, in \cite{Goldstein1973}, Goldstein used similar ideas to analyse $C_0$-groups of isometries on Orlicz spaces (see also \cite{Fleming1976} where a slight inaccuracy in Goldstein's paper was corrected). Moreover, we note that in \cite{Lyubich1997} Lyubich presented a unified approach to the spectral analysis of operators with invariant convex sets in finite dimensions; in particular, this approach allows the analysis of contractions and the analysis of positive operators within the same framework. \smallskip \par

Let us give a short outline of the current article: We mainly focus on contractive $C_0$-semi\-groups $(e^{tA})_{t \ge 0}$ (and a slight generalisation of them, see Definition \ref{def_asymptotic_contractivity}), but we also give several results for single operators which are derived from the semigroup results. In Section \ref{section_peripheral_point_spectrum_on_extremely_non_hilbert_spaces} we consider Banach spaces that do not isometrically contain a two-dimensional Hilbert space. This is a very strong condition, which enables us to prove strong results on the point spectrum of semigroups and operators. Some classical sequence spaces will turn out to be typical examples of such Banach spaces. In Section \ref{section_peripheral_point_spectrum_on_projectively_non_hilbert_spaces} we weaken the geometric condition on the Banach space to allow for a wider range of spaces. Instead we will impose additional assumptions on our semigroups to prove similar results on the point spectrum. In Section \ref{section_peripheral_spectrum_in_capitel_l_p} we employ an ultra power technique to deduce results on the spectrum rather than on the point spectrum. To make this approach work we need geometrical assumptions not only on the Banach space $X$, but also on an ultra power $X_\calU$ of $X$. It turns out that those assumptions are fulfilled for $L^p$-spaces if $p \not\in \{1,2,\infty\}$. \smallskip \par

We point out that the main reason why our results are interesting is that we do not impose any positivity assumptions on our semigroups. In the case of positive semigroups most of our results can be derived from the well-known spectral theory of positive semigroups (see \cite[Sections~C-III and~C-IV]{Arendt1986} for an overview; see also the more recent articles \cite{Davies2005} for a treatment of positive contractive semigroups on $l^p$, \cite{Keicher2006, Wolff2008} for positive semigroups on more general atomic spaces and \cite{Arendt2008, Gerlach2013} for the non-atomic case). \smallskip \par 

Before starting our analysis, let us fix some notation that will be used throughout the article: By $\bbN := \{1,2,3,...\}$ we denote the set of strictly positive integers and we set $\bbN_0 := \bbN \cup \{0\}$. Whenever $X$ is a vector space over a field $\bbF$ and $M \subset X$, then we denote by $\operatorname{span}_\bbF(M)$ the \emph{linear span} of $M$ in $X$ over $\bbF$. Whenever $X$ is a real or complex Banach space, the set of bounded linear operators on $X$ will be denoted by $\calL(X)$. For a complex Banach space $X$ and a closed linear operator $A: X \supset D(A) \to X$ its \emph{spectrum}, \emph{point spectrum} and \emph{approximate point spectrum} are denoted respectively by $\sigma(A)$, $\sigma_{\operatorname{pnt}}(A)$ and $\sigma_{\operatorname{appr}}(A)$; by $s(A) := \sup \{\re \lambda: \lambda \in \sigma(A)\}$ we denote the \emph{spectral bound} of $A$. If $A \in \calL(X)$, then $r(A)$ denotes the \emph{spectral radius} of $A$. Whenever $\lambda \not\in \sigma(A)$, then $R(\lambda, A) := (\lambda - A)^{-1}$ is the \emph{resolvent} of $A$. For each $\lambda \in \bbC$, we denote by $\operatorname{Eig}(\lambda, A) := \ker(\lambda - A)$ the \emph{eigenspace} of $A$ corresponding $\lambda$; this notation will be used even if $\lambda \not\in \sigma_{\operatorname{pnt}}(A)$, and in this case we have $\operatorname{Eig}(\lambda,A) = \{0\}$, of course. If $X$ is a real or complex Banach space, then we denote by $X'$ its \emph{norm dual space}, and if $A: X \supset D(A) \to X$ is a densely defined linear operator, then we denote by $A'$ the \emph{adjoint operator} of $A$. We will always use the symbol $\bbT$ to denote the unit circle $\bbT := \{\lambda \in \bbC: |\lambda| = 1\}$. All measure spaces considered in this paper are allowed to be non-$\sigma$-finite, unless something else is stated explicitly.

Let us briefly discuss the concept of \emph{complexifications} of real Banach spaces which we will need throughout the article. Complexifications of real Banach spaces are discussed in some detail in \cite{Munoz1999}; there the authors consider the notion of a \emph{reasonable complexification} which is a bit more restrictive than our definition of a complexification below (but which would also be fine for our purposes in this paper). \par 
Let $X$ be a real Banach space. By a \emph{complexification} of $X$ we mean a tuple $(X_\bbC,J)$, where $X_\bbC$ is a complex Banach space and $J: X \to X_\bbC$ is an isometric $\bbR$-linear mapping with the following properties:
\begin{itemize}
	\item[(a)] $X_\bbC = J(X) \oplus i \, J(X)$, meaning that $J(x) \cap i\,J(X) = \{0\}$ and $X_\bbC = J(X) + i \, J(X)$. 
	\item[(b)] The ($\bbR$-linear) projection from $X_\bbC$ onto $J(X)$ along $i\,J(X)$ is contractive.
\end{itemize}	

It is often convenient to identify $X$ with its image $J(X)$ and to shortly say that $X_\bbC \supset X$ is a complexification of $X$, thereby suppressing the mapping $J$ in the notation. If $X_\bbC \supset X$ is a complexification of a real Banach space $X$, then we can decompose each element $z \in X_\bbC$ into a sum $z = \re z + i\im z$, where $\re z$ and $\im z$ are uniquely determined elements of $X$; as usual we denote by $\overline{z} := x - iy$ the \emph{complex conjugate vector} of $z$. It follows from property (b) above that the mappings $\re$ and $\im$ are contractive. \par 
Note that if $X_\bbC$ is a complexification of $X$, then every linear operator $A: X \supset D(A) \to X$ has a unique $\bbC$-linear extension $A_\bbC: X_\bbC \supset D(A_\bbC) := D(A) + iD(A) \to X$; we call $A_\bbC$ the \emph{complex extension} of $A$ to the complexification $X_\bbC$ of $X$. Similarly, every bounded linear functional $x' \in X'$ has a unique $\bbC$-linear extension $(x')_\bbC: X_\bbC \to \bbC$; the functional $(x')_\bbC$ is also bounded and therefore an element of the dual space $(X_\bbC)'$ of $X_\bbC$. Let $H: X' \to (X_\bbC)'$, $H(x') = (x')_\bbC$; using that $\re$ and $\im$ are contractive one can show that $H$ is isometric and this implies that $((X_\bbC)',H)$ is a complexification of $X'$, i.e.\ we can consider the dual space of a complexification of $X$ as a complexification of the dual space of $X$. It is then easy to verify that the adjoint $(A_\bbC)'$ of the complex extension of a densely defined linear operator $A: X \supset D(A) \to X$ coincides with the complex extension of the adjoint $A'$. \par 
We point out that \emph{every} real Banach space $X$ has a complexification $X_\bbC$. For example, we can simply endow an algebraic complexification $X_\bbC$ of $X$ with the norm $||x+iy|| := \sup_{\theta \in [0,2\pi]} ||x \cos \theta + y \sin \theta ||$. This particular complexification has the advantage that we have $\|T\| = \|T_\bbC\|$ for every operator $T \in \calL(X)$, a property which need not be true on general complexifications. However, the reader should be warned that on many concrete function spaces the norm defined above does not coincide with the ``natural norm'' one would usually endow the corresponding complex function space with. \par 
In general there are, of course, more than one complexifications of a real Banach space $X$; two complexifications of $X$ need not be isometric, but they are always topologically $\bbC$-linearly isomorphic via a (uniquely determined) canonical mapping which acts as the identity on $X$. Using this observation it is easy to see that many properties of complex extensions of operators do not depend on the choice of the complexification of the Banach space. In particular, if $A_{\bbC,1}$ and $A_{\bbC,2}$ denote complex extensions of an operator $A: X \supset D(A) \to X$ to two complexifications $X_{\bbC,1}$ and $X_{\bbC,2}$, respectively, then we have $\sigma(A_{\bbC,1}) = \sigma(A_{\bbC,2})$ and $\sigma_{\operatorname{pnt}}(A_{\bbC,1}) = \sigma_{\operatorname{pnt}}(A_{\bbC,2})$. Therefore it makes sense to define the \emph{spectrum} $\sigma(A)$ (respectively, the \emph{point spectrum} $\sigma_{\operatorname{pnt}}(A)$) of $A$ to be the spectrum (respectively, the point spectrum) of the complex extension $A_\bbC$ of $A$ to any complexification $X_\bbC$ of $X$.

\section{The point spectrum on extremely non-Hilbert spaces} \label{section_peripheral_point_spectrum_on_extremely_non_hilbert_spaces}

\subsection{Extremely non-Hilbert spaces}

When we consider contractive $C_0$-semi\-groups $(e^{tA})_{t \ge 0}$ on a real Banach space $X$, it turns out that the existence of purely imaginary eigenvalues of (the complex extension of) $A$ is related to the existence of two-dimensional Hilbert spaces in $X$, see Theorem \ref{theorem_point_spectrum_for_sg_on_extremely_non_hilbert_spaces} below. This motivates the following notion.

\begin{definition}
	A real Banach space $X$ is called \emph{extremely non-Hilbert} if it does not isometrically contain a two-dimensional Hilbert space.
\end{definition}

Clearly, any closed subspace of an extremely non-Hilbert space is extremely non-Hilbert itself. Before we analyse the spectral properties of contractive semigroups on extremely non-Hilbert spaces, we want to provide some examples of those spaces to give the reader an idea of when our subsequent results are applicable. 

\begin{example} \label{example_l_p_sequence_space_is_extremely_non_hilbert_for_p_not_even}
	(a) Let $1 \le p < \infty$, $p \not\in 2 \bbN$. Then the real sequence space $l^p := l^p(\bbN;\bbR)$ is extremely non-Hilbert (see \cite[Corollary 1.8]{Delbaen1998}). For finite dimensional $l^p$-spaces the same result was shown earlier in \cite[Proposition 1]{Lyubich1970}. \par
	(b) If $p \in 2\bbN$, then $l^p$ is not extremely non-Hilbert. This is obvious for $p = 2$, and for $p \in \{4,6,...\}$ it can be shown that even a finite dimensional $(\bbR^n, ||\cdot||_p)$ exists which isometrically contains a two-dimensional Hilbert space. This is proved for example in \cite[p.\,283--284]{Milman1988} (note however that the necessary dimension $n$ increases with $p$). See also \cite[the proof of Proposition 2]{Lyubich1970}, \cite{Lyubich1993} and \cite{Konig2004} for a discussion of this and related topics.
\end{example}

In contrast to $l^p$-sequence spaces, real-valued $L^p$-spaces on non-discrete measure spaces are not extremely non-Hilbert, in general. For $L^p$-spaces on the torus we will give an explicit example of an isometrically embedded two-dimensional Hilbert space in Example \ref{example_cpaital_l_p_not_extremly_non_hilbert_shift_sg}. Moreover, it can be shown by probabilistic methods that $L^p([0,1];\bbR)$ even isometrically contains the infinite dimensional sequence space $l^2$, see \cite[p.\,16]{Johnson2001}. \par
A further example of an extremely non-Hilbert space is the space $c := c(\bbN;\bbR)$ of real-valued convergent sequences on $\bbN$. This follows from the next, more general example:

\begin{example} \label{examp_c_k_is_extremely_non_hilbert_for_countable_k}
	Let $K$ be a compact Hausdorff space and let $C(K;\bbR)$ be the space of continuous, real-valued functions on $K$ which is endowed with the supremum norm. If $K$ is countable, then $C(K,\bbR)$ is extremely non-Hilbert.
	\begin{proof}
		Assume for a contradiction that $V \subset C(K;\bbR)$ is a two-dimensional Hilbert space. Since $V$ is isometrically isomorphic to $(\bbR^2, ||\cdot||_2)$, there are two vectors $x,y \in V$ such that
		\begin{align*}
			\varphi: [0,2\pi] \to V \text{,} \quad t \mapsto \cos(t) x + \sin(t) y
		\end{align*}
		is a mapping into the unit sphere of $V$. Hence, the mapping $t \mapsto ||\varphi(t)|| = \max_{k \in K} |\varphi(t)(k)|$ is identically $1$. \par
		Now, for each $k \in K$, let $A_k := \{t \in [0,2\pi]: |\varphi(t)(k)| = 1\}$. Each set $A_k$ is closed and we have $\bigcup_{k \in K} A_k = [0,2\pi]$. Since $K$ is countable, Baire's Theorem implies that at least one of the sets $A_k$, say $A_{k_0}$, has non-empty interior and thus contains a non-empty open interval $I$. For continuity reasons, the mapping
		\begin{align*}
			t \mapsto \varphi(t)(k_0) = \cos(t)x(k_0) + \sin(t)y(k_0)
		\end{align*}
		is either identically $1$ or identically $-1$ on $I$. Hence, it is identically $1$ or identically $-1$ on the whole interval $[0,2\pi]$, due to the Identity Theorem for analytic functions. This is a contradiction since $\cos$, $\sin$ and $\mathbbm{1}$ (denoting the function on $[0,2\pi]$ which is identically $1$) are linearly independent functions on $[0,2\pi]$.
	\end{proof}
\end{example}

Preceding to the above example, we claimed that the sequence space $c$ is extremely non-Hilbert. This follows indeed from Example \ref{examp_c_k_is_extremely_non_hilbert_for_countable_k} since $c$ is isometrically isomorphic to the space $C(\bbN \cup \{\infty\};\bbR)$, where $\bbN \cup \{\infty\}$ is the one-point-compactification of the discrete space $\bbN$. Note that this also implies that the space $c_0 := c_0(\bbN;\bbR)$ of all real-valued sequences which converge to $0$ is extremely non-Hilbert, since it is a closed subspace of $c$.

\subsection{$C_0$-semigroups on extremely non-Hilbert spaces} \label{subsection_c_0_semigroups_on_extremely_non_hilbert_spaces}

Since we have now several examples of extremely non-Hilbert spaces at hand, let us turn to contractive semigroups on them. In fact, we will not really need our semigroups to be contractive. In each of our theorems, one of the following asymptotic properties will suffice.

\begin{definition} \label{def_asymptotic_contractivity}
	Let $(e^{tA})_{t \ge 0}$ be a $C_0$-semigroup on a real Banach space $X$. The semigroup $(e^{tA})_{t \ge 0}$ is called
	\begin{itemize}
		\item[(a)] \emph{uniformly asymptotically contractive} if $\limsup_{t \to \infty} ||e^{tA}|| \le 1$. \par
		\item[(b)] \emph{strongly asymptotically contractive} if $\limsup_{t \to \infty} ||e^{tA}x|| \le 1$ for each $x \in X$ with $||x|| = 1$. \par
		\item[(c)] \emph{weakly asymptotically contractive} if $\limsup_{t \to \infty} | \langle e^{tA}x, x' \rangle | \le 1$ for all $x \in X$ and all $x' \in X'$ with $||x|| = ||x'|| = 1$.
	\end{itemize}
	For every operator $T \in \calL(X)$, the same notions are defined by replacing the semigroup by the powers $T^n$.
\end{definition}

If a $C_0$-semigroup $(e^{tA})_{t \ge 0}$ on $X$ is weakly asymptotically contractive, then it is bounded, and if an operator $T \in \calL(X)$ is weakly asymptotically contractive, then it is power-bounded; this follows from the Uniform Boundedness Principle. \par

We now begin our analysis with a result on the purely imaginary eigenvalues of the semigroup generator. 

\begin{theorem} \label{theorem_point_spectrum_for_sg_on_extremely_non_hilbert_spaces}
	Let $X$ be a real Banach space which is extremely non-Hilbert and let $(e^{tA})_{t \ge 0}$ be a weakly asymptotically contractive $C_0$-semigroup on $X$. Then $\sigma_{\operatorname{pnt}}(A) \cap i \bbR \subset \{0\}$.
	\begin{proof}
		Let $X_\bbC$ be a complexification of $X$. The complex extension $A_\bbC$ of $A$ generates a $C_0$-semigoup $(e^{tA_\bbC})_{t \ge 0}$ on $X_\bbC$, where each operator $e^{tA_\bbC}$ is the complex extension of the operator $e^{tA}$; by definition $\sigma_{\operatorname{pnt}}(A)$ is the point spectrum of $A_\bbC$. \par 
		Assume for a contradiction that $i\alpha$ is an eigenvalue of $A_\bbC$, where $\alpha \in \bbR \setminus \{0\}$. Then we have $A_\bbC z = i\alpha z$ for some $0 \not= z = x+iy \in X_\bbC$, where $x,y \in X$. One easily checks that $A_\bbC \overline z = -i\alpha \overline{z}$ and that $x$ and $y$ are linearly independent over $\bbR$. In particular, $x \not= 0$ and thus we may assume that $||x|| = 1$. \par
		Let $V$ be the linear span of $x$ and $y$ over $\bbR$. We show that $V$ a Hilbert space with respect to the norm induced by $X$. For each $t \ge 0$ we have
		\begin{align*}
			e^{tA}x = e^{tA_\bbC}\re z = \re (e^{tA_\bbC}z)  = \operatorname{Re}(e^{i\alpha t} z) = \cos(\alpha t)x - \sin(\alpha t)y \text{.}
		\end{align*}
		This shows that the orbit $(e^{tA}x)_{t \ge 0}$ is periodic. Together with the weak asymptotic contractivity of $(e^{tA})_{t\ge 0}$ this implies that $||e^{tA}x|| = 1$ for each $t \ge 0$.	Hence, the vectors $\cos(\alpha t)x - \sin(\alpha t)y$ are contained in the surface of the unit ball of $V$. Now, endow $\bbR^2$ with the euclidean norm $||\cdot||_2$ and consider the linear bijection $\phi: \bbR^2 \to V$, $w \mapsto w_1x - w_2 y$. Whenever $||w||_2 = 1$, the vector $w$ can be written as $w = (\cos(\alpha t), \sin(\alpha t))$ for some $t \ge 0$. Hence, $\phi(w) = \cos(\alpha t)x - \sin(\alpha t)y$ and thus, $\phi$ maps the surface of the unit ball of $(\bbR^2,||\cdot||_2)$ into the surface of the unit ball of $V$. This shows that $\phi$ is isometric, i.e.\ $V$ is a Hilbert space.
	\end{proof}
\end{theorem}

It follows from Example \ref{example_l_p_sequence_space_is_extremely_non_hilbert_for_p_not_even} that the dual or the pre-dual of an extremely non-Hilbert space need not be extremely non-Hilbert, in general. For example, the sequence space $l^{\frac{4}{3}}$ is extremely non-Hilbert, while its dual and pre-dual space $l^{4}$ is not. Nevertheless, the assertion of the above Theorem \ref{theorem_point_spectrum_for_sg_on_extremely_non_hilbert_spaces} can be extended to reflexive Banach spaces with an extremely non-Hilbert dual space:

\begin{corollary} \label{cor_point_spectrum_for_sg_on_spaces_with_extremely_non_hilbert_dual_spaces}
	Let $X$ be a reflexive Banach space over the real field and suppose that its dual space $X'$ is extremely non-Hilbert. Let $(e^{tA})_{t \ge 0}$ be a weakly asymptotically contractive $C_0$-semigroup on $X$. Then $\sigma_{\operatorname{pnt}}(A) \cap i \bbR \subset \{0\}$.
	\begin{proof}
		Since $X$ is reflexive, the semigroup of adjoint operators $((e^{tA})')_{t \ge 0}$ is a $C_0$-semigroup on $X'$ and its generator is the adjoint $A'$ of $A$. The reflexivity of $X$ also implies that the adjoint semigroup $((e^{tA})')_{t \ge 0}$ is weakly asymptotically contractive, and thus we conclude that $\sigma_{\operatorname{pnt}}(A') \cap i \bbR \subset \{0\}$ by Theorem \ref{theorem_point_spectrum_for_sg_on_extremely_non_hilbert_spaces}. However, we have $\sigma_{\operatorname{pnt}}(A') \cap i \bbR \supset \sigma_{\operatorname{pnt}}(A) \cap i\bbR$ according to \cite[Lemma~2.3]{Arendt1988} since the semigroup $(e^{tA})_{t \ge 0}$ is bounded. This proves the corollary.
	\end{proof}
\end{corollary}

Recall again that an operator $A$ on a real Banach space $X$ is said to have \emph{compact resolvent} if its complex extension $A_\bbC$ to a complexification $X_\bbC$ of $X$ has compact resolvent, a property which does not depend on the choice of $X_\bbC$. For semigroups whose generator has compact resolvent, Theorem \ref{theorem_point_spectrum_for_sg_on_extremely_non_hilbert_spaces} and Corollary \ref{cor_point_spectrum_for_sg_on_spaces_with_extremely_non_hilbert_dual_spaces} yield the following convergence result.

\begin{corollary} \label{cor_convergence_for_compact_resolvent_on_extremely_non_hilbert_space}
	Suppose the assumptions of Theorem \ref{theorem_point_spectrum_for_sg_on_extremely_non_hilbert_spaces} or of Corollary \ref{cor_point_spectrum_for_sg_on_spaces_with_extremely_non_hilbert_dual_spaces} are fulfilled. If $A$ has compact resolvent, then $e^{tA}$ strongly converges as $t \to \infty$.
	\begin{proof}
		We know from Theorem \ref{theorem_point_spectrum_for_sg_on_extremely_non_hilbert_spaces} respectively from Corollary \ref{cor_point_spectrum_for_sg_on_spaces_with_extremely_non_hilbert_dual_spaces} that $\sigma_{\operatorname{pnt}}(A) \cap i \bbR \subset \{0\}$. Hence the assertion follows from \cite[Theorem V.2.14 and Corollary V.2.15]{Engel2000}.
	\end{proof}
\end{corollary}

Let us mention the particular case of $l^p$-spaces in an extra corollary:

\begin{corollary} \label{cor_peripheral_point_spectrum_and_convergence_on_l_p_sequence_spaces}
	Let $(e^{tA})_{t \ge 0}$ be a strongly continuous semigroup on $l^p = l^p(\bbN; \bbR)$ for $1 \le p < \infty$ and $p \not= 2$. If $(e^{tA})_{t\ge 0}$ is weakly asymptotically contractive, then $\sigma_{\operatorname{pnt}}(A) \cap i \bbR \subset \{0\}$. \par
	If, in addition, $A$ has compact resolvent, then $e^{tA}$ strongly converges as $t \to \infty$.
	\begin{proof}
		If $p$ is not an even integer, then $l^p$ is an extremly non-Hilbert space by Example \ref{example_l_p_sequence_space_is_extremely_non_hilbert_for_p_not_even} and thus the assumptions of Theorem \ref{theorem_point_spectrum_for_sg_on_extremely_non_hilbert_spaces} are fulfilled. If $p > 2$ is an even integer instead, then the conjugate index $p'$, which is defined by $\frac{1}{p} + \frac{1}{p'} = 1$, is not an even integer. Thus, the dual space $(l^p)' = l^{p'}$ is extremely non-Hilbert and so the assumptions of Corollary \ref{cor_point_spectrum_for_sg_on_spaces_with_extremely_non_hilbert_dual_spaces} are fulfilled.
	\end{proof}
\end{corollary}

The following connection of Corollary \ref{cor_peripheral_point_spectrum_and_convergence_on_l_p_sequence_spaces} to the spectral theory of positive semigroups is interesting:

\begin{remark}
	It is shown in \cite[Theorem 9]{Davies2005} that a contractive positive $C_0$-semigroup on $l^p$ for $1 \le p < \infty$ always fulfils $\sigma_{\operatorname{pnt}}(A) \cap i \bbR \subset \{0\}$. Corollary \ref{cor_peripheral_point_spectrum_and_convergence_on_l_p_sequence_spaces} shows that the assumption of positivity is not needed for this result whenever $p \not= 2$. Note however, that the result for positive semigroups can be generalized to a large class of semigroups on other Banach lattices, as well (see the articles quoted in the introduction for details). 
\end{remark}

Now, let us demonstrate by a couple of examples what goes wrong with our above results on spaces which are not extremely non-Hilbert. First, we consider an example of a contraction semigroup on a two-dimensional Hilbert-space:

\begin{example}
	Consider the two-dimensional euclidean space $(\bbR^2,||\cdot||_2)$ and the $C_0$-semigroup $(e^{tA})_{t \ge 0}$, where the matrices of $A$ and $e^{tA}$ are given by
	\begin{align*}
		A =
		\begin{pmatrix}
			0 & -1 \\
			1 & 0
		\end{pmatrix}
		\quad \text{and} \quad
		e^{tA} =
		\begin{pmatrix}
			\cos t & -\sin t \\
			\sin t & \cos t
		\end{pmatrix}
		\text{.}
	\end{align*}
	Then $e^{tA}$ acts as a rotation with angle $t$ on $\bbR^2$, and $(e^{tA})_{t \ge 0}$ is clearly a contraction semigroup with respect to the euclidean norm $||\cdot||_2$. Since $(\bbR^2,||\cdot||_2)$ is a Hilbert space, we cannot apply Theorem \ref{theorem_point_spectrum_for_sg_on_extremely_non_hilbert_spaces}, and indeed the spectrum of the generator $A$ is given by $\sigma(A) = \{i,-i\}$, i.e.\ the assertion of Theorem \ref{theorem_point_spectrum_for_sg_on_extremely_non_hilbert_spaces} fails in our example. \par 
	It is also very instructive to observe what happens on $(\bbR^2,||\cdot||_p)$ for $p \not= 2$. If we consider the same semigroup as above on such a space, then we still have $\sigma(A) = \{i,-i\}$, but the space is extremely non-Hilbert now (for $p\not\in 2\bbN$ this follows from Example \ref{example_l_p_sequence_space_is_extremely_non_hilbert_for_p_not_even}, but it is true even for $p \in 2\bbN \setminus \{2\}$ since our space is two-dimensional). Hence, another assumption of Theorem \ref{theorem_point_spectrum_for_sg_on_extremely_non_hilbert_spaces} must fail here, and indeed, it is easily seen that our rotation semigroup is not contractive with respect to the $||\cdot||_p$-norm due to the low symmetry of the unit ball of this norm.
\end{example}

Our next example shows that Theorem \ref{theorem_point_spectrum_for_sg_on_extremely_non_hilbert_spaces} does in general not hold on $L^p$-spaces on non-discrete measure spaces.

\begin{example} \label{example_cpaital_l_p_not_extremly_non_hilbert_shift_sg}
	Let $1 \le p < \infty$, let the complex unit circle $\bbT$ be equipped with a non-zero Haar measure $\mu$ and let $(e^{tA})_{t \ge 0}$ be the shift semigroup on $L^p(\bbT;\bbR)$, i.e.\ $e^{tA}f(z) = f(e^{it}z)$ for all $f \in L^p(\bbT;\bbR)$. Then $(e^{tA})$ is contractive, but the point spectrum of $A$ is given by $\sigma_{\operatorname{pnt}}(A) = i \bbZ$. \par 
	By virtue of Theorem \ref{theorem_point_spectrum_for_sg_on_extremely_non_hilbert_spaces} this implies that $L^p(\bbT;\bbR)$ is not extremely non-Hilbert; however we can also see this explicitly. In fact, let $f_1: \bbT \to \bbR$, $f_1(z) = \re z$ and $f_2: \bbT \to \bbR$, $f_2(z) = \im z$. Let $\alpha, \beta \in \bbR$ and choose $\varphi \in [0,2\pi)$ and $r \ge 0$ such that $\alpha - i\beta = re^{i\varphi}$. Then we obtain
	\begin{align*}
		||\alpha f_1 + \beta f_2||_p = & \Big[ \int_\bbT \big|\re ((\alpha - i \beta) z) \big|^p \, |dz|\Big]^{\frac{1}{p}} = \Big[ \int_\bbT \big|\re (re^{i\varphi} z) \big|^p \, |dz|\Big]^{\frac{1}{p}} = \\
		= & \, r \cdot \Big[ \int_\bbT \big|\re z\big|^p \, |dz| \Big]^{\frac{1}{p}} = (\alpha^2 + \beta^2)^{\frac{1}{2}} \cdot ||f_1||_p \text{.}
	\end{align*}
	Hence, $(\alpha, \beta) \mapsto \frac{1}{||f_1||_p}(\alpha f_1 + \beta f_2)$ yields an isometric isomorphism between the two-dimensional euclidean space $(\bbR^2, ||\cdot||_2)$ and the subspace $\operatorname{span}_\bbR \{ f_1,f_2 \}$ of $L^p(\bbT;\bbR)$. The reader might also compare \cite[Remark 5]{Goldstein1973} for a slightly different presentation of this example.
\end{example}

Finally, we want to demonstrate that it is essential for our above results (and also for the results in the subsequent sections) that we consider $C_0$-semigroups of \emph{real} operators:

\begin{example}
	Consider the space $(\bbC^n, ||\cdot||_p)$ for some arbitrary $p \in [1,\infty]$ and let $A \in \bbC^{n \times n}$ be the diagonal matrix whose diagonal entries are all equal to $i$. Then $(e^{tA})_{t \ge 0}$ is clearly a contractive $C_0$-semigroup on $(\bbC^n, ||\cdot||_p)$, but we obtain $\sigma(A) = \{i\}$ for the spectrum of its generator.
\end{example}

\subsection{Single operators on extremely non-Hilbert spaces}

Next, we show how Theorem \ref{theorem_point_spectrum_for_sg_on_extremely_non_hilbert_spaces} can be applied to yield a result on the point spectrum of a single operator.

\begin{theorem} \label{theorem_point_spectrum_for_operator_on_extremely_non_hilbert_space}
	Let $X$ be a real Banach space which is extremely non-Hilbert and let $T \in \calL(X)$ be weakly asymptotically contractive. Then $\sigma_{\operatorname{pnt}}(T) \cap \bbT$ consists only of roots of unity.
	\begin{proof}
		Let $T_\bbC$ be the complex extension of $T$ to a complexification $X_\bbC$ of $X$. Assume for a contradiction that $e^{i\alpha} \in \sigma_{\operatorname{pnt}}(T) \cap \bbT$ (where $\alpha \in \bbR$) is not a root of unity and let $T_\bbC z = e^{i\alpha} z$ for some $0 \not= z = x + iy \in X_\bbC$, where $x,y \in X$. It is easy to show that $x$ and $y$ are linearly independent over $\bbR$ and that we have $T_\bbC \, \overline{z} = e^{-i\alpha} \, \overline{z}$. \par
		Let $V = \operatorname{span}_\bbR\{x,y\}$ and $V_\bbC = \operatorname{span}_\bbC\{z,\overline{z}\}$. Then we have $V_\bbC = V \oplus iV$, and therefore $V_\bbC$ is a complexification of $V$.
		We define a $C_0$-semigroup $(e^{tA_\bbC})_{t \ge 0}$ on $V_\bbC$ by means of
		\begin{align*}
			e^{tA_\bbC}z = e^{i t}z\text{,} \quad e^{tA_\bbC} \, \overline{z} = e^{- i t} \, \overline{z} \text{.}
		\end{align*}
		Note that $e^{tA_\bbC}$ is well-defined since $z$ and $\overline{z}$ are linearly independent over $\bbC$ (as they are eigenvectors of $T_\bbC$ for two different eigenvalues). Clearly, this semigroup is strongly (in fact uniformly) continuous and and its generator $A_\bbC$ is the complexification of an operator $A \in \calL(V)$, given by
		\begin{align*}
			Ax = -y \quad \text{and} \quad Ay = x \text{.}
		\end{align*}
		We have $\sigma_{\operatorname{pnt}}(A) = \{-i,i\}$, and the semigroup $(e^{tA_\bbC})_{t \ge 0}$ is the complexification of the semigroup $(e^{tA})_{t \ge 0}$. We now show that $(e^{tA})_{t \ge 0}$ is a contraction semigroup on $V$, so that we can apply Theorem \ref{theorem_point_spectrum_for_sg_on_extremely_non_hilbert_spaces}: \par 
		Let $t \ge 0$. Since $e^{i\alpha}$ is not a root of unity, we find a sequence of integers $N_n \to \infty$ such that $e^{iN_n\alpha} \to e^{it}$. Since $V \subset V_\bbC$, each element $v \in V$ can be written in the form $v = \lambda_1 z + \lambda_2 \overline{z}$ with complex scalars $\lambda_1,\lambda_2$. We thus obtain for each $x' \in X'$ that 
		\begin{align*}
			|\langle e^{tA}v, x' \rangle| = & |\langle e^{tA_\bbC}v, \, (x')_\bbC \rangle| = |\langle \lambda_1 e^{it} z + \lambda_2 e^{-it} \overline{z} , \, (x')_\bbC \rangle| = \\
			= & \lim_n | \langle \lambda_1 T_\bbC^{N_n}z + \lambda_2 T_\bbC^{N_n} \overline{z} , \, (x')_\bbC \rangle | = \lim_n |\langle T^{N_n} v , x' \rangle| \le |\langle v , x' \rangle| \text{.}
		\end{align*}
		Thus, each operator $e^{tA}$ is contractive. Since $X$ is extremely non-Hilbert, so is its subspace $V$, and we conclude from Theorem \ref{theorem_point_spectrum_for_sg_on_extremely_non_hilbert_spaces} that $\sigma_{\operatorname{pnt}}(A) \cap i \bbR \subset \{0\}$, which is a contradiction.
	\end{proof}
\end{theorem}

Similar to Corollary \ref{cor_point_spectrum_for_sg_on_spaces_with_extremely_non_hilbert_dual_spaces}, we obtain the following dual result:

\begin{corollary} \label{theorem_point_spectrum_for_operator_on_a_space_with_extremely_non_hilbert_dual_space}
	Let $X$ be a reflexive Banach space over the real field and suppose that its dual space $X'$ is extremely non-Hilbert. Let $T \in \calL(X)$ be weakly asymptotically contractive. Then $\sigma_{\operatorname{pnt}}(T) \cap \bbT$ consists only of roots of unity.
	\begin{proof}
		Since $X$ is reflexive, the adjoint operator $T'$ is also weakly asymptotically contractive, so $\sigma_{\operatorname{pnt}}(T') \cap \bbT$ only consists of roots of unity by Theorem \ref{theorem_point_spectrum_for_operator_on_extremely_non_hilbert_space}. Moreover, since $T$ is power-bounded, we can show with the same proof as in \cite[Lemma~2.3]{Arendt1988} that $\sigma_{\operatorname{pnt}}(T') \cap \bbT \supset \sigma_{\operatorname{pnt}}(T) \cap \bbT$. This proves the corollary.
	\end{proof}
\end{corollary}

Again, we note that the sequence spaces $l^p$ for $1 \le p < \infty$ and $p\not= 2$ fulfil the assumption of either Theorem \ref{theorem_point_spectrum_for_operator_on_extremely_non_hilbert_space} or Corollary \ref{theorem_point_spectrum_for_operator_on_a_space_with_extremely_non_hilbert_dual_space}. \par

Under additional assumptions on $T$, Theorem \ref{theorem_point_spectrum_for_operator_on_extremely_non_hilbert_space} and Corollary \ref{theorem_point_spectrum_for_operator_on_a_space_with_extremely_non_hilbert_dual_space} can be used to derive results on the asymptotic behaviour of the powers $T^n$. However, we delay this to Subsection \ref{subsection_single_operators_on_projectively_non_hilbert_spaces}. There we first prove additional spectral results on single operators on another type of Banach spaces; then we describe the asymptotic behaviour of $T^n$ in Corollary \ref{cor_periodicity_of_compact_operators_on_projectively_non_hilbert_space}.

\section{The point spectrum on projectively non-Hilbert spaces} \label{section_peripheral_point_spectrum_on_projectively_non_hilbert_spaces}

\subsection{Projectively non-Hilbert spaces}

The results of Section \ref{section_peripheral_point_spectrum_on_extremely_non_hilbert_spaces} all have the major flaw that they are applicable only on the small range of Banach spaces which are extremely non-Hilbert or have an extremely non-Hilbert dual space. We have seen in Example \ref{example_cpaital_l_p_not_extremly_non_hilbert_shift_sg} that $L^p$-spaces on non-discrete measure spaces do, in general, not fulfil this property. This indicates that we should consider Banach spaces which fulfil only a weaker geometric condition (at the cost of extra conditions on our semigroup, of course). The geometric condition in the subsequent Definition \ref{def_projectively_non_hilbert_space} seems to be the appropriate notion to do this. Before stating this definition, let us recall that a subspace $V$ of a (real or complex) Banach space $X$ is said to \emph{admit a contractive linear projection in $X$} if there is a projection $P \in \calL(X)$ such that $PX = V$ and such that $||P|| \le 1$.

\begin{definition} \label{def_projectively_non_hilbert_space}
	A real Banach space $X$ is called \emph{projectively non-Hilbert} if it does not isometrically contain a two-dimensional Hilbert space $V \subset X$ which admits a contractive linear projection in $X$.
\end{definition}

Note that if $X$ is a projectively non-Hilbert space and $P \in \calL(X)$ is a contractive projection, then the range $PX$ is a projectively non-Hilbert space, too. It was shown by Lyubich in \cite{Lyubich1970} that projectively non-Hilbert spaces provide an appropriate setting for the spectral analysis of finite dimensional (and more generally, for compact) contractive operators. Preceding to his work, Krasnosel'ski{\u\i} achieved similar results in \cite{Krasnoselskii1968}, but on a more special class of Banach spaces which he called \emph{completely non-Hilbert spaces}. \par 
Of course, a Banach space which is extremely non-Hilbert is also projectively non-Hilbert, but in fact the class of projectively non-Hilbert spaces is much larger. Again, we start with several examples for those spaces.

\begin{example} \label{ex_capital_l_p_spaces_as_projectively_non_hilbert_spaces}
	Let $(\Omega, \Sigma, \mu)$ be a measure space and let $1 \le p < \infty$. If $p \not= 2$, then the real Banach space $L^p := L^p(\Omega, \Sigma, \mu; \bbR)$ is projectively non-Hilbert. \par
	To see this, let $V \subset L^p$ be a two-dimensional subspace and assume that $P$ is a contractive linear projection on $L^p$ with range $P L^p = V$. Then $V$ is isometrically isomorphic to $L^p(\tilde \Omega, \tilde \Sigma, \tilde \mu; \bbR)$ for another measure space $(\tilde \Omega, \tilde \Sigma, \tilde \mu)$, see \cite[Theorem 6]{Tzafriri1969}. However, every two-dimensional real-valued $L^p$-space is isometrically isomorphic to the space $(\bbR^2, \|\cdot\|_p)$. Since $p \not= 2$, it follows that $V$ cannot be a Hilbert space.
\end{example}

We saw in Example \ref{example_cpaital_l_p_not_extremly_non_hilbert_shift_sg} that the space $L^p(\bbT;\bbR)$ on the complex unit circle $\bbT$ isometrically contains a two-dimensional Hilbert space $V$. The above Example \ref{ex_capital_l_p_spaces_as_projectively_non_hilbert_spaces} shows that this subspace cannot admit a contractive projection (although it of course admits a bounded projection since it is finite dimensional). To obtain further examples of projectively non-Hilbert spaces, we now show how they behave with respect to duality.

\begin{proposition} \label{prop_dual_spaces_projectively_non_hilbert}
	Let $X$ be a real Banach space. If its dual space $X'$ is projectively non-Hilbert, then so is $X$ itself.
\end{proposition}

This proposition shows that projectively non-Hilbert spaces behave better with respect to duality than extremely non-Hilbert spaces do (compare the comments before Corollary \ref{cor_point_spectrum_for_sg_on_spaces_with_extremely_non_hilbert_dual_spaces}). The proof of Proposition \ref{prop_dual_spaces_projectively_non_hilbert} relies on the following elementary observation:

\begin{lemma} \label{lem_projections_and_duality}
	Let $P$ be a bounded linear projection on a (real or complex) Banach space $X$. Then the mapping
	\begin{align*}
		i: \quad P' \, X' & \; \to \; (PX)' \text{,} \\
		x' & \; \mapsto \; x'|_{PX}
	\end{align*}
	is a contractive Banach space isomorphism. If $P$ is contractive, then $i$ is even isometric.
	\begin{proof}
		Clearly, $i$ is linear and contractive. If $x' \in P'X'$ and $i(x') = x'|_{PX} = 0$, then we obtain for every $x \in X$ that
		\begin{align*}
			\langle x', x \rangle = \langle P'x', x \rangle = \langle x', Px \rangle = 0 \text{,} 
		\end{align*}
		hence $x' = 0$. This shows that $i$ is injective. Surjectivity of $i$ follows from the Hahn-Banach Theorem: A given functional $\tilde x' \in (PX)'$ can be extended to a functional $x' \in X'$ and we obtain for each $x \in PX$ that
		\begin{align*}
			\langle P'x', x \rangle = \langle x', Px \rangle = \langle x', x \rangle = \langle \tilde x', x \rangle \text{.}
		\end{align*}
		Thus, $i(P'x') = (P'x')|_{PX} = \tilde x'$, i.e.\ $i$ is surjective and hence a Banach space isomorphism. \par
		Finally, assume that the projection $P$ is a contraction and let $x' \in P'X'$. For each $\varepsilon > 0$, we can find a normalized vector $x \in X$ such that $|\langle x', x \rangle| \ge ||x'|| - \varepsilon$. Since $||Px|| \le ||x|| = 1$, we obtain
		\begin{align*}
			||i(x')|| = || x'|_{PX} || \ge |\langle x'|_{PX}, Px \rangle| = |\langle x', x \rangle| \ge ||x'|| - \varepsilon \text{.}
		\end{align*}
		Thus, $||i(x')|| \ge ||x'||$. Since $i$ is a contraction, we conclude that it is in fact isometric.
	\end{proof}
\end{lemma}

\begin{proof}[Proof of Proposition \ref{prop_dual_spaces_projectively_non_hilbert}]
	If $X$ is not projectively non-Hilbert, then there is a two-dimensional Hilbert space $V \subset X$ and a contractive projection $P$ from $X$ onto $V$. The adjoint operator $P'$ is also a contractive projection and by Lemma \ref{lem_projections_and_duality}, its range $P'X'$ is isometrically isomorphic to $(PX)' = V'$ which is a two-dimensional Hilbert space. Thus, $X'$ is not projectively non-Hilbert, either.
\end{proof}

From Proposition \ref{prop_dual_spaces_projectively_non_hilbert} we obtain another class of examples of projectively non-Hilbert spaces:

\begin{example}
	Let $L$ be a locally compact Hausdorff space and let $C_0(L;\bbR)$ be the space of real-valued continuous functions on $L$ which vanish at infinity, endowed with the supremum norm. Then $C_0(L; \bbR)$ is projectively non-Hilbert. \par
	To see this, note that the dual space $C_0(L;\bbR)'$ is isometrically isomorphic to $L^1(\Omega, \Sigma, \mu; \bbR)$ for some measure space $(\Omega, \Sigma, \mu)$. This well-known result follows from Kakutani's Representation Theorem for abstract $L$-spaces, see \cite[Proposition 1.4.7 (i) and Theorem 2.7.1]{Meyer-Nieberg1991}. Since we know from Example \ref{ex_capital_l_p_spaces_as_projectively_non_hilbert_spaces} that $L^1(\Omega, \Sigma, \mu; \bbR)$ is projectively non-Hilbert, Proposition \ref{prop_dual_spaces_projectively_non_hilbert} implies that $C_0(L;\bbR)$ is projectively non-Hilbert as well.
\end{example}

If $K$ is a compact Hausdorff space, then we of course have $C_0(K;\bbR) = C(K;\bbR)$, and so the space $C(K;\bbR)$ of continuous, real-valued functions on $K$ is projectively non-Hilbert. This also implies that the space $L^\infty(\Omega, \Sigma, \mu; \bbR)$, for an arbitrary measure space $(\Omega, \Sigma, \mu)$, is projectively non-Hilbert. In fact, each $L^\infty$-space is isometrically isomorphic to a $C(K;\bbR)$-space for some compact Hausdorff space $K$, due to Kakutani's Representation Theorem for abstract $M$-spaces (see \cite[Theorem 2.1.3]{Meyer-Nieberg1991}).

Let us close this subsection with some further examples of projectively non-Hilbert spaces:

\begin{example}
	(a) As explained in \cite[the last paragraph of Remark~3.5]{Randrianantoanina2004}, Orlicz sequences spaces which satisfy certain technical conditions are projectively non-Hilbert. \par 
	(b) If $p,q \in (1,\infty) \setminus \{2\}$, then the vector-valued sequence space $l^p(l^q)$, more precisely being given by $l^p(l^q) := l^p(\bbN; l^q(\bbN;\bbR))$, is projectively non-Hilbert. This can easily be derived from \cite[Theorem~5.1]{Lemmens2007}.
\end{example}

\subsection{$C_0$-semigroups on projectively non-Hilbert spaces}

We are now going to study the point spectrum of contractive $C_0$-semigroups on projectively non-Hilbert spaces. Before stating the main result of this section, we recall that a $C_0$-semigroup $(e^{tA})_{t\ge 0}$ on a (real or complex) Banach space $X$ is called \emph{weakly almost periodic}, if for each $x \in X$, the orbit $\{e^{tA}x: t \ge 0\}$ is weakly pre-compact in $X$. 

\begin{theorem} \label{theorem_point_spectrum_for_sg_on_projectively_non_hilbert_spaces}
	Let $X$ be real Banach space which is projectively non-Hilbert and let $(e^{tA})_{t \ge 0}$ be a weakly almost periodic $C_0$-semigroup on $X$. If $(e^{tA}x)_{t\ge 0}$ is weakly asymptotically contractive and if $\sigma_{\operatorname{pnt}}(A) \cap i \bbR$ is bounded, then we actually have $\sigma_{\operatorname{pnt}}(A) \cap i \bbR \subset \{0\}$.
\end{theorem}

Before proving Theorem \ref{theorem_point_spectrum_for_sg_on_projectively_non_hilbert_spaces}, let us make a few remarks on the assumptions and consequences of the theorem: The condition on $\sigma_{\operatorname{pnt}}(A) \cap i \bbR$ to be bounded is for example fulfilled if $(e^{tA})_{t \ge 0}$ is eventually norm continuous, since in this case the intersection of the entire spectrum $\sigma(A)$ with every right half plane $\{\lambda \in \bbC: \re \lambda \ge \alpha \}$ (where $\alpha \in \bbR$) is automatically bounded (see \cite[Theorem II.4.18]{Engel2000}). If $(e^{tA})_{t \ge 0}$ is bounded, then the condition that the semigroup be weakly almost periodic is automatically fulfilled if the space $X$ is reflexive. It is also automatically fulfilled if the semigroup is bounded and eventually compact (i.e.\ if $e^{tA}$ is compact for sufficiently large $t$) or if it is bounded and its generator has compact resolvent; in the latter two cases the orbits of the semigroup are even pre-compact in the norm topology on $X$, see \cite[Corollary V.2.15]{Engel2000}. \par

Moreover, for eventually compact semigroups we even obtain the following convergence result as a corollary:

\begin{corollary} \label{cor_convergence_for_compact_resolvent_on_projectively_non_hilbert_spaces}
	Let $X$ be a real Banach space which is projectively non-Hilbert and let $(e^{tA})_{t \ge 0}$ be an eventually compact $C_0$-semigroup on $X$ which is weakly asymptotically contractive. Then $e^{tA}$ converges with respect to the operator norm as $t \to \infty$.
	\begin{proof}
		Let $A_\bbC$ be the complex extension of $A$ to a complexification $X_\bbC$ of $X$. Since the semigroup $(e^{tA_\bbC})_{t \ge 0}$ is eventually compact, it is weakly almost periodic (in fact, its orbits are even pre-compact in norm, see \cite[Corollary V.2.15 (ii)]{Engel2000}). Moreover, as an eventually compact semigroup, $(e^{tA_\bbC})_{t \ge 0}$ is eventually norm-continuous (cf.\ \cite[Lemma II.4.22]{Engel2000}) and therefore, the intersection of $\sigma(A)$ with any right half plane $\{z \in \bbC: \re z \ge \alpha\}$, $\alpha \in \bbR$, is bounded. We thus can apply Theorem \ref{theorem_point_spectrum_for_sg_on_projectively_non_hilbert_spaces} to conclude that $\sigma_{\operatorname{pnt}}(A) \cap i\bbR \subset \{0\}$. Since the spectrum of $A_\bbC$ consists only of eigenvalues (see \cite[Corollary V.3.2 (i), and assertion (ii) at the end of paragraph IV.1.17]{Engel2000}), this implies that $\sigma(A) \cap i\bbR \subset \{0\}$. \par
		Since the semigroup $(e^{tA_\bbC})_{t \ge 0}$ is eventually compact, it admits the representation stated in \cite[Corollary~V.3.2, formula~(3.1)]{Engel2000}. Using that $\sigma(A) \cap i\bbR \subset \{0\}$ and that the semigroup is bounded, it follows that $e^{tA_\bbC}$ converges with respect to the operator norm as $t \to \infty$.
	\end{proof}
\end{corollary}

Next, we want to prove Theorem \ref{theorem_point_spectrum_for_sg_on_projectively_non_hilbert_spaces}. Since we are going to reuse several parts of the proof in Section \ref{section_peripheral_spectrum_in_capitel_l_p}, we extract those parts into two lemmas. The first lemma is based on an idea used by Lyubich in the proof of \cite[Theorem 1]{Lyubich1970}.

\begin{lemma} \label{lem_rotation_sg_with_one_eigenvalue_implies_contr_proj_hilbert_space}
	Let $X_\bbC$ be a complexification of a real Banach space $X$ and let $(e^{tA})_{t \ge 0}$ be a $C_0$-semigroup of isometries on $X$. Furthermore, suppose that we have $\sigma_{\operatorname{pnt}}(A) = \{i,-i\}$ and that $X_\bbC = Y_1 \oplus Y_2$, where $Y_1 := \operatorname{Eig}(i,A_\bbC)$ and $Y_2 := \operatorname{Eig}(-i,A_\bbC)$. Then $X$ cannot be projectively non-Hilbert.
	\begin{proof}
		We define a binary operation $\circ: \bbC \times X_\bbC \to X_\bbC$ by
		\begin{align*}
			\lambda \circ (y_1 + y_2) = \lambda y_1 + \overline{\lambda} y_2 \quad \text{for all} \quad y_1 \in Y_1, \; y_2 \in Y_2 \text{.}
		\end{align*}
		The mapping $\circ$ is continuous and moreover it is easy to check that $(X_\bbC, +, \circ)$ is a complex vector space again. For each $t \ge 0$, each $r \ge 0$ and each $x = y_1 + y_2 \in X_\bbC$ (where $y_1 \in Y_1$, $y_2 \in Y_2$) we have
		\begin{align*}
			(re^{it}) \circ x = r e^{it}y_1 + r e^{-it}y_2 = r \cdot \big( e^{tA_\bbC} y_1 + e^{tA_\bbC} y_2 \big) = r e^{tA_\bbC} x \text{.}
		\end{align*}
		This implies that $X$ is invariant under the \emph{complex} multiplication $\circ$, i.e.\ $X$ is a \emph{complex} vector subspace of $(X_\bbC,+,\circ)$. \par
		Moreover, we have for each $t \ge 0$, each $r \ge 0$ and each $x \in X$ that $||(re^{it}) \circ x|| = ||re^{tA_\bbC}x|| = ||re^{tA}x|| = r \, ||x||$ since $e^{tA}$ is isometric. Hence, $||\lambda \circ x|| = |\lambda| \, ||x||$ for each $\lambda \in \bbC$ and each $x \in X$. Therefore, the norm $||\cdot||$ on our real Banach space $X$ is also a norm on the complex vector space $(X,+,\circ)$. \par 
		Now, choose an element $x_0 \in X$ such that $||x_0|| = 1$. The one-dimensional complex subspace $\bbC \circ x_0 \subset X$ admits a contractive $\bbC$-linear (with respect to $\circ$) projection $P: X \to \bbC \circ x_0$ due to the Hahn-Banach Theorem. Moreover, $P$ is also a linear mapping over the real field with respect to the original multiplication on $X$ since we have $r\circ x = rx$ for each $r \in \bbR$ and each vector $x \in X$. Also note that $(\bbC \circ x_0, +, \circ)$ is two-dimensional over the real field, and since the restriction of $\circ$ to real scalars coincides with the original multiplication on $X$, $\bbC \circ x_0$ is a two dimensional real vector subspace of the original space $X$. \par
		Finally, we show that $\bbC \circ x_0$ is a real Hilbert space. Since $||x_0|| = 1$, the mapping $\psi: \bbC \to  \bbC \circ x_0$, $\lambda \mapsto \lambda \circ x_0$ is a $\bbC$-linear (with respect to $\circ$) and isometric bijection. If we restrict all scalars to the real field, $\psi$ thus becomes an isometric and linear bijection between the two-dimensional normed real spaces $(\bbR^2, ||\cdot||_2)$ and $(\bbC \circ x_0, ||\cdot||)$. This proves the lemma.
	\end{proof}
\end{lemma}

The next lemma is an elementary fact from linear algebra.

\begin{lemma} \label{lem_eigenvectors_for_square_roots_of_eigenvalues}
	Let $T$ be a linear operator on a complex vector space $X$ and let $\lambda \in \bbC \setminus \{0\}$. Then $\operatorname{Eig} (\lambda, T^2) = \operatorname{Eig} (\mu, T) \oplus \operatorname{Eig} (-\mu, T)$, where $\mu$ and $-\mu$ denote the complex square roots of $\lambda$.
	\begin{proof}
		First note that we have $\operatorname{Eig} (\mu, T) \cap \operatorname{Eig} (-\mu, T) = \{0\}$ since $-\mu \not= \mu$; hence, the sum is indeed direct. \par 
		The inclusion ``$\supset$'' is obvious. To show the converse inclusion, let $x \in \ker (\lambda - T^2)$; we may assume $x \not= 0$. Since $T^2x = \lambda x$, the linear span $V = \operatorname{span}_\bbC \{ x, Tx \}$ is a $T$-invariant vector subspace of $X$. As $T^2x = \lambda x$ and $T^2 \, Tx = \lambda Tx$, we have $(T|_V)^2 = T^2|_V = \lambda\id_V$. Since $\lambda \not= 0$, the operator $T|_V$ cannot be similar to a non-trivial Jordan block; this is obvious if $\dim V = 1$ and it follows from $(T|_V)^2 = \lambda \id_V$ and a short matrix computation if $\dim V = 2$. Hence $V$ contains a basis of eigenvectors of $T|_V$. From the spectral mapping theorem we know that $\sigma(T|_V) \subset \{-\mu,\mu\}$ and thus $V = \operatorname{span}_\bbC \{ x, Tx \}$ is spanned by eigenvectors of $T$ with corresponding eigenvalues $-\mu$ and/or $\mu$. In particular, $x$ is a linear combination of such eigenvectors.
	\end{proof}
\end{lemma}

Of course it may happen that one of the spaces $\operatorname{Eig} (\mu, T)$ and $\operatorname{Eig} (-\mu, T)$ in the above lemma equals $\{0\}$ while the other does not. Moreover, note that Lemma~\ref{lem_eigenvectors_for_square_roots_of_eigenvalues} is false for $\lambda = 0$; of course, the sum is no longer direct in this case, but the assertions even fails if we do not require the sum to be direct. Simply choose $T$ as a two-dimensional Jordan block with eigenvalue $0$ to see this. \par
We are now ready to prove Theorem \ref{theorem_point_spectrum_for_sg_on_projectively_non_hilbert_spaces}. For the proof, recall that an operator $T$ on a Banach space $X$ is called \emph{mean ergodic} if the sequence of \emph{Ces{\`a}ro means} $\frac{1}{n}\sum_{k=0}^{n-1} T^k$ is strongly convergent as $n \to \infty$. In this case it is well-known (and easy to see) that the limit is a projection $P \in \calL(X)$ whose range coincides with the fixed space $\ker(1-T)$ of $T$; $P$ is called the \emph{mean ergodic projection} of $T$.

\begin{proof}[Proof of Theorem \ref{theorem_point_spectrum_for_sg_on_projectively_non_hilbert_spaces}]
	(a) Let $X_\bbC$ be a complexification of $X$ and let $(e^{tA_\bbC})_{t \ge 0}$ be the complex extension of the semigroup $(e^{tA})_{t \ge 0}$. Assume for a contradiction that $\sigma_{\operatorname{pnt}}(A) \cap i\bbR \not \subset \{0\}$. Replacing $A$ by $c A$ for some $c > 0$, we may assume that $i,-i \in \sigma_{\operatorname{pnt}}(A)$ and that $\sigma_{\operatorname{pnt}}(A) \subset i \cdot (-3, 3)$. \par
	(b) We proceed with some observations about several eigenspaces that will be used throughout the proof. For each $t \in (0,\frac{\pi}{2}]$ we have that
	\begin{align}
		\operatorname{Eig}(e^{it}, e^{tA_\bbC}) \; \; & = \overline{\operatorname{span}}_{n \in \bbZ}\operatorname{Eig} ( \; \, i + \frac{2\pi in}{t},A_\bbC) \; = \operatorname{Eig}(\, i, A_\bbC) \; \; =: Y_1 \label{form_time_independent_eigenspace_plus}  \\
		\operatorname{Eig}(e^{-it}, e^{tA_\bbC}) & = \overline{\operatorname{span}}_{n \in \bbZ}\operatorname{Eig} (-i + \frac{2\pi in}{t},A_\bbC) = \operatorname{Eig}(-i, A_\bbC) =: Y_2 \text{.}  \label{form_time_independent_eigenspace_minus}
	\end{align}
	The equalities on the left follow from a general relationship between the eigenspaces of $A_\bbC$ and $e^{tA_\bbC}$ (see \cite[Corollary IV.3.8]{Engel2000}) and hold for all $t > 0$; the equalities on the right hold for all $t \in (0,\frac{\pi}{2}]$ due to the condition $\sigma_{\operatorname{pnt}}(A) \subset i \cdot (-3, 3)$. \par
	While the equalities (\ref{form_time_independent_eigenspace_plus}) and (\ref{form_time_independent_eigenspace_minus}) hold only for $t \in (0, \frac{\pi}{2}]$, we have for $t \ge 0$ at least the inclusions
	\begin{align}
		Y_1 \subset \operatorname{Eig}(e^{it}, e^{tA_\bbC}) \quad \text{and} \quad Y_2 \subset \operatorname{Eig}(e^{-it}, e^{tA_\bbC}). \label{form_inclusions_for_eigenspaces_for_large_t}
	\end{align}
	For the time $t = \pi$ we can make a more precise observation: In fact, we have
	\begin{align*}
		\operatorname{Eig}(-1, e^{\pi A_\bbC}) = \operatorname{Eig}(e^{i\frac{\pi}{2}}, e^{\frac{\pi}{2} A_\bbC}) \oplus \operatorname{Eig}(e^{-i\frac{\pi}{2}}, e^{\frac{\pi}{2} A_\bbC}) = Y_1 \oplus Y_2 =: Z_\bbC
	\end{align*}
	The first of these equalities follows if we apply Lemma \ref{lem_eigenvectors_for_square_roots_of_eigenvalues} to the operator $T = e^{\frac{\pi}{2} A_\bbC}$ and to the complex number $\lambda = -1$; the second equality follows from (\ref{form_time_independent_eigenspace_plus}) and (\ref{form_time_independent_eigenspace_minus}). \par
	
	(c) Let us analyse the space $Z_\bbC$. Since the operator $-e^{\pi A_\bbC}$ is almost weakly periodic (meaning that the orbits of its powers are weakly pre-compact in $X$), it is mean ergodic \cite[Proposition~1.1.19]{Emelyanov2007}. Hence the Ces{\`a}ro means $\frac{1}{n}\sum_{k=0}^{n-1} (-e^{\pi A_\bbC})^k$ converge to a projection $P_\bbC: X_\bbC \to Z_\bbC$, because $Z_\bbC$ is the fixed space of $-e^{\pi A_\bbC}$. Clearly, $P_\bbC$ leaves the real space $X$ invariant and therefore, the range $Z_\bbC = P_\bbC X_\bbC$ is the complexification of the real space $Z := P_\bbC X = X \cap Z_\bbC$. Moreover, the restriction $P_\bbC|_X$ is contractive, since the operator $-e^{\pi A} = -e^{\pi A_\bbC}|_X$ is weakly asymptotically contractive. Thus, the real Banach space $Z$ is the range of the contractive projection $P_\bbC|_X \in \calL(X)$ and is therefore projectively non-Hilbert. \par
	 
	 (d) Finally, we want to apply Lemma \ref{lem_rotation_sg_with_one_eigenvalue_implies_contr_proj_hilbert_space} to the space $Z_\bbC = Y_1 \oplus Y_2$ and to the restricted semigroup $(e^{tA_\bbC}|_{Z_\bbC})_{t \ge 0}$ in order to obtain a contradiction. First, note that $(e^{tA_\bbC}|_{Z_\bbC})_{t \ge 0}$ is indeed a semigroup on $Z_\bbC$, since $e^{tA_\bbC}$ leaves $Z_\bbC$ invariant due to (\ref{form_inclusions_for_eigenspaces_for_large_t}). Moreover, the spectral assumptions of Lemma \ref{lem_rotation_sg_with_one_eigenvalue_implies_contr_proj_hilbert_space} are clearly fulfilled. Since the semigroup $(e^{tA_\bbC}|_{Z_\bbC})_{t \ge 0}$ is periodic and since it is weakly asymptotically contractive on $Z \subset X$, we conclude that it acts in fact isometrically on $Z$. Thus, the assumptions of Lemma \ref{lem_rotation_sg_with_one_eigenvalue_implies_contr_proj_hilbert_space} are fulfilled and we can conclude from this lemma that $Z$ is not projectively non-Hilbert. This contradicts the statement proved in (c).
\end{proof}

A crucial argument in step (c) of the preceding proof is the use of a mean ergodic theorem to obtain a contractive projection onto the real part of $Y_1 \oplus Y_2$. This argument stems from the proof of \cite[Theorem 1]{Lyubich1970}. In this context, we should also mention that the condition on $(e^{tA})_{t \ge 0}$ to be weakly almost periodic can be slightly relaxed in Theorem \ref{theorem_point_spectrum_for_sg_on_projectively_non_hilbert_spaces}: Indeed, the proof of the theorem shows that we only need that each negative operator $-e^{tA}$ is mean ergodic. However, we preferred to state the theorem with the condition that $(e^{tA})_{t \ge 0}$ be weakly almost periodic, since this seems to be more natural than a condition on the negative operators $-e^{tA}$. \par
Note that Theorem \ref{theorem_point_spectrum_for_sg_on_projectively_non_hilbert_spaces} fails if we do not require the set $\sigma_{\operatorname{pnt}}(A) \cap i \bbR$ a priori to be bounded. A counter example is again provided by the shift semigroup on $L^p(\bbT;\bbR)$.

\subsection{Single operators on projectively non-Hilbert spaces} \label{subsection_single_operators_on_projectively_non_hilbert_spaces}

As in Section \ref{section_peripheral_point_spectrum_on_extremely_non_hilbert_spaces}, we now apply our semigroup result to study the single operator case. Recall that an operator $T$ on a (real or complex) Banach space $X$ is called \emph{weakly almost periodic} if for each $x \in X$, the orbit $\{T^nx: n\in \bbN_0\}$ is weakly pre-compact in $X$.

\begin{theorem} \label{theorem_point_spectrum_of_op_on_projectively_non_hilbert_space}
	Let $X$ be a real Banach space which is projectively non-Hilbert and let $T \in \calL(X)$ be weakly almost periodic. If $T$ is weakly asymptotically contractive and if $\sigma_{\operatorname{pnt}}(T) \cap \bbT$ is finite, then $\sigma_{\operatorname{pnt}}(T) \cap \bbT$ in fact only consists of roots of unity.
	\begin{proof}
		(a) By means of the well-known Jacobs-deLeeuw-Glicksberg decomposition (see e.g.~\cite[Section~2.4]{Krengel1985} and \cite[Section~V.2]{Engel2000}) we can find a projection $P: X \to X$ which reduces $T$ and has the following properties: the restriction of $T$ to $PX$ is bijective, $\sigma_{\operatorname{pnt}}(T|_{PX}) = \sigma_{\operatorname{pnt}}(T) \cap \bbT$ and any complexification $Y_\bbC$ of $PX$ is the closed linear span of all eigenvectors of $(T|_{PX})_{\bbC}$. Since $T$ is weakly asymptotically contractive, it follows that $P$ is contractive and that $T|_{PX}$ is isometric. In particular, $PX$ is projectively non-Hilbert. \par 
		We may thus assume for the rest of the proof that $T$ is an isometric bijection with $\sigma_{\operatorname{pnt}}(T) \subset \bbT$ and that any complexification $X_\bbC$ of $X$ is the closed linear span of all eigenvectors of the complex extension $T_\bbC$ of $T$. Since the point spectrum of $T_\bbC$ is finite and $T_\bbC$ is power bounded, we can show by a similar technique as in the proof of \cite[Theorem~1.1]{Markus1971} that the projection from the direct sum $\bigoplus_{\lambda \in \sigma_{\operatorname{pnt}}(T)} \operatorname{Eig}(\lambda, T_\bbC) \subset X_\bbC$ onto each eigenspace $\operatorname{Eig}(\lambda,T_\bbC)$ along the other eigenspaces is continuous. Hence, the direct sum is closed, which shows that actually $X_\bbC = \bigoplus_{\lambda \in \sigma_{\operatorname{pnt}}(T)} \operatorname{Eig}(\lambda, T_\bbC)$. \smallskip \par
		(b) We may choose a finite set $\calA \subset \bbR$ such that the exponential function maps $2\pi i \calA$ bijectively to $\sigma_{\operatorname{pnt}}(T)$. Let $\calB \subset \calA \cup \{1\}$ be a basis of $\operatorname{span}_\bbQ(\calA \cup \{1\})$ over the field of rational numbers $\bbQ$ which fulfils $1 \in \calB$. For each $\beta \in \calB$ and each $\alpha \in \calA$, denote by $\alpha_\beta \in \bbQ$ the uniquely determined rational number such that $\alpha = \sum_{\beta \in \calB} \alpha_\beta \beta$. We can find an integer $k \not= 0$ such that $k\alpha_\beta \in \bbZ$ for each $\alpha \in \calA$ and each $\beta \in \calB$. Now, assume for a contradiction that at least one element of $\calA$ is not contained in $\bbQ$. Then $\calB$ also contains an irrational number $\beta_0$. \smallskip \par
		(c) We define a strongly (in fact uniformly) continuous semigroup $(e^{tA})_{t \ge 0}$ on $X_\bbC$ by means of
		\begin{align*}
			e^{tA_\bbC}x_\alpha = e^{it k \alpha_{\beta_0}}x_\alpha \quad \text{for } x_\alpha \in \operatorname{Eig}(e^{2\pi i \alpha}, T_\bbC) \text{.}
		\end{align*}
		We show that this semigroup leaves $X$ invariant and that its restriction to $X$ is contractive: Let $t \ge 0$. Since $\calB$ is linearly independent over $\bbQ$ and since $1 \in \calB$, we conclude from Kronecker's Theorem that the powers of the tuple $(e^{2\pi i \beta})_{\beta \in \calB \setminus \{1\}}$ are dense in $\bbT^{\calB \setminus \{1\}}$. Hence, we can find a sequence of natural numbers $(N_n)_{n \in \bbN}$ such that $e^{2\pi i N_n \beta_0} \to e^{i t}$ and such that $e^{2\pi i N_n \beta} \to 1$ for each $\beta \in \calB \setminus \{\beta_0, 1\}$. Moreover, we clearly have $e^{2\pi i N_n \cdot 1} \to 1$, so that $e^{2\pi i N_n \beta} \to 1$ holds actually for each $\beta \in \calB \setminus \{\beta_0\}$. Therefore, we obtain for each $\alpha \in \calA$ and $x_\alpha \in \operatorname{Eig}(e^{2\pi i \alpha}, T_\bbC)$ that
		\begin{align*}
			T_\bbC^{k N_n} x_\alpha & = e^{2\pi i kN_n \alpha} x_\alpha =  \prod_{\beta \in \calB} e^{2\pi i k N_n \alpha_\beta \beta} x_\alpha = \\
			& = \prod_{\beta \in \calB} \big( e^{2\pi i N_n \beta} \big)^{k \alpha_\beta} x_\alpha \overset{n \to \infty}{\to} e^{itk \alpha_{\beta_0}} x_\alpha = e^{tA_\bbC}x_\alpha \text{,}
		\end{align*}
		which implies $e^{tA_\bbC}x = \lim_n T_\bbC^{kN_n}x$ for each $x \in X_\bbC$. This shows that $e^{tA_\bbC}$ leaves $X$ invariant and that it is contractive on $X$. We conclude that $A_\bbC$ is the complex extension of an operator $A$ on $X$ and that $A$ generates a contractive $C_0$-semigroup $(e^{tA})_{t \ge 0}$ on $X$ whose complex extension is given by $(e^{tA_\bbC})_{t \ge 0}$. \smallskip \par 
		(d) For the point spectrum of $A$ we have $\sigma_{\operatorname{pnt}}(A) = \{ik \alpha_{\beta_0}: \alpha \in \calA\}$. In particular, the element $0 \not= ik = ik (\beta_0)_{\beta_0}$ is contained in $\sigma_{\operatorname{pnt}}(A)$. However, we can apply Theorem \ref{theorem_point_spectrum_for_sg_on_projectively_non_hilbert_spaces} to the semigroup $(e^{tA})_{t \ge 0}$: Indeed, the semigroup $(e^{tA_\bbC})_{t \ge 0}$ is weakly almost periodic, since each trajectory $\{e^{tA_\bbC} x: t \ge 0\}$ is bounded and contained in a finite-dimensional space. Hence, $(e^{tA})_{t \ge 0}$ is almost weakly periodic, and Theorem \ref{theorem_point_spectrum_for_sg_on_projectively_non_hilbert_spaces} implies that $\sigma_{\operatorname{pnt}}(A) \cap i \bbR \subset \{0\}$. This is a contradiction.
	\end{proof}
\end{theorem}

For compact operators, Theorem \ref{theorem_point_spectrum_of_op_on_projectively_non_hilbert_space} was proved by Lyubich in \cite{Lyubich1970} by similar methods (but more directly, since the paper \cite{Lyubich1970} focussed on single operators rather than on $C_0$-semigroups). Moreover, on some important function spaces, Theorem \ref{theorem_point_spectrum_of_op_on_projectively_non_hilbert_space} was proved for compact operators even earlier by Krasnosel'ski{\u\i} in \cite{Krasnoselskii1968}. 

Under some additional a priori assumptions on the spectrum, we obtain the following simple corollary concerning the asymptotics of $(T^n)_{n \in \bbN_0}$:

\begin{corollary} \label{cor_periodicity_of_compact_operators_on_projectively_non_hilbert_space}
	Let $X$ be a real Banach space which is projectively non-Hilbert and suppose that $T \in \calL(X)$ is weakly asymptotically contractive. Assume furthermore that $\sigma(T) \cap \bbT$ is finite, isolated from the rest of the spectrum and consists only of poles of the resolvent $R(\cdot,T_\bbC)$. Then $T$ can be decomposed into two linear operators $T = T_{\operatorname{per}} + T_0$ such that $||T_0^n|| \to 0$ as $n \to \infty$ and $T_{\operatorname{per}}^{n_0+1} = T_{\operatorname{per}}$ for some $n_0 \ge 1$.
	\begin{proof}
		Let $T_\bbC$ denote the complex extension of $T$ to a complexification $X_\bbC$ of $X$ and let $P_\bbC$ be the spectral projection corresponding to the part $\sigma(T) \cap \bbT$ of the spectrum of $T_\bbC$. Then the powers of the operator $T_{\bbC, 0} := T_\bbC(1- P_\bbC)$ converge to $0$ with respect to the operator norm. \par
		Moreover, each $\lambda \in \sigma(T) \cap \bbT$ is a simple pole of the resolvent. Indeed, if this was not true, then we could find a vector $z \in \ker\big((\lambda-T)^2\big) \setminus \ker(\lambda-T)$ (this follows from \cite[Theorems~1 and~2 in Section VIII.8]{Yosida1995}. A brief induction argument then shows that $T^nz = n\lambda^{n-1}(T-\lambda)z + \lambda^n z$ for all $n \in \bbN_0$, contradicting the fact the $T$ is power-bounded. Thus, $\lambda$ is indeed a first order pole of the resolvent and hence, the range of the spectral projection corresponding to $\lambda$ coincides with the eigenspace $\ker(\lambda-T)$ \cite[Theorem~3 in Section~VIII.8]{Yosida1995}.
		
		Since $P_\bbC$ is the sum of all spectral projections corresponding to the single spectral values in $\bbT$, it follows that for each $x \in X_\bbC$ the vector $P_\bbC x$ can be written as a finite sum $P_\bbC x = x_1 + ... + x_m$, where $x_1$, ..., $x_m$ are all eigenvectors of $T$ belonging to eigenvalues in $\bbT$. Hence, the bounded set $\{T_\bbC^nP_\bbC x: n \in \bbN_0\}$ is contained in the finite dimensional space $\operatorname{span}_\bbC\{x_1,...,x_m\}$ and is thus pre-compact. Since for each $\varepsilon > 0$, the set
		\begin{align*}
			\{T_\bbC^n(1-P_\bbC)x: n \in \bbN_0\} = \{\big(T_\bbC (1-P_\bbC) \big)^nx: n\in \bbN_0\} \text{.}
		\end{align*}
		contains only finitely many elements whose norm is large than $\varepsilon$, we conclude that this set is pre-compact, as well. This implies that the trajectory $\{T_\bbC^nx: n \in \bbN_0\}$ is pre-compact. In particular, $T_\bbC$ is weakly almost periodic, and so is $T$. \par
		Hence, we can apply Theorem \ref{theorem_point_spectrum_of_op_on_projectively_non_hilbert_space} to conclude that $\sigma_{\operatorname{pnt}}(T) \cap \bbT$ only consists of roots of unity. Therefore, the powers of the operator $T_{\bbC, \operatorname{per}} := T_\bbC P_\bbC$ are periodic. Finally, note that the spectral projection $P_\bbC$ leaves $X$ invariant (this easily follows from the representation of $P_\bbC$ by Cauchy's Integral Formula), and hence the operators $T_{\bbC,0}$ and $T_{\bbC, \operatorname{per}}$ leave $X$ invariant, as well. Thus, the assertion follows with $T_0 := T_{\bbC, 0}|_X$ and $T_{\operatorname{per}} := T_{\bbC, \operatorname{per}}|_X$.
	\end{proof}
\end{corollary}

Note that the assumptions on $\sigma(T) \cap \bbT$ in Corollary \ref{cor_periodicity_of_compact_operators_on_projectively_non_hilbert_space} are for example fulfilled if $T$ is compact or, more generally, if the essential spectral radius $r_{\operatorname{ess}}(T_\bbC) := \sup\{|\lambda|: \lambda - T_\bbC \text{ is not Fredholm}\}$ is strictly smaller than $1$. \par

In finite dimensions a result related to Corollary~\ref{cor_periodicity_of_compact_operators_on_projectively_non_hilbert_space} can also be found in~\cite[Theorem~2.1]{Lemmens2003}. Besides this, we want to mention that a very similar result holds even for non-linear operators if they satisfy an additional assumption on their \emph{$\omega$-limit sets}; this was shown by Lemmens and van Gaans in \cite[Theorem 2.8]{Lemmens2009}, employing a result on projectively non-Hilbert spaces from \cite[Theorem 4]{Lyubich1970}. For the special case of $L^p$-spaces, $1 < p < \infty$, $p \not= 2$, such a non-linear result had already been shown by Sine in \cite[Theorem 3]{Sine1990}.

\section{The spectrum on $L^p$-spaces} \label{section_peripheral_spectrum_in_capitel_l_p}

In the preceding sections, we considered $C_0$-semigroups $(e^{tA})_{t \ge 0}$ and focussed on results that ensure that $A$ has no purely imaginary eigenvalues. However, to turn those results into convergence results, we always needed rather strong compactness conditions. \par
In this final section, we will employ an ultra-power technique to ensure that the intersection of the entire spectrum with $i \bbR$ is trivial. This will allow us to derive a much more general convergence result. To make our method work, we have to ensure that an ultra-power of our space $X$ is still a projectively non-Hilbert space. This seems to be rather difficult in general, but $L^p$-spaces are particularly well-suited for this task since an ultra-power of an $L^p$-space is again an $L^p$-space.

\subsection{Ultra powers of Banach spaces}

We shortly recall the most important facts about ultra powers of Banach spaces that will be used in the next subsection. For a more detailed treatment, we refer for example to \cite{Heinrich1980}, \cite[p.\,251--253]{Meyer-Nieberg1991} and \cite[Section V.1]{Schaefer1974}. Let $X$ be a real or complex Banach space. By $l^\infty(X) := l^\infty(\bbN,X)$ we denote the space of all bounded sequences in $X$. We endow this space with the supremum norm $||x||_\infty = \sup_{n \in \bbN} ||x_n||$ for $x = (x_n)_{n \in \bbN} \in l^\infty(X)$. \par
Let $\calU$ be a free ultra filter on $\bbN$ and define $c_{0,\calU}(X) := \{x \in l^\infty(X): \lim_\calU x_n = 0\}$. The space $c_{0,\calU}(X)$ is a closed vector subspace of $l^\infty(X)$, and the quotient space $X_\calU := l^\infty(X) / c_{0,\calU}(X)$ is called the \emph{$\calU$-ultra power} of $X$. For each $x = (x_n)_{n \in \bbN} \in l^\infty(X)$, we denote by $x_\calU$ the equivalence class $x_\calU := x + c_{0,\calU}(X) \in X_\calU$. It turns out that we can compute the quotient norm on the space $X_\calU$ rather easily: For each $x_\calU \in X_\calU$, we have $||x_\calU||_{X_\calU} = \lim_\calU ||x_n||$ (for complex Banach spaces $X$, this can be found e.g.\ in \cite[Theorem 4.1.6]{Meyer-Nieberg1991} or \cite[Proposition V.1.2]{Schaefer1974}; for real Banach spaces, the proof is the same). Note that this limit always exists since $\calU$ is an ultra filter and since the set $\{||x_n||: n \in \bbN\}$ is pre-compact in $\bbR$. \par
If $T \in \calL(X)$, we can define an operator $\tilde T \in \calL(l^\infty(X))$ by the pointwise operation
\begin{align*}
	\tilde T x = \tilde T (x_n)_{n \in \bbN} := (Tx_n)_{n \in \bbN} \quad \text{for all } x = (x_n)_{n \in \bbN} \in l^\infty(X) \text{.}
\end{align*}

The bounded linear operator $\tilde T$ leaves the subspace $c_{0,\calU}(X) \subset l^\infty(X)$ invariant and thus induces another bounded linear operator $T_\calU \in \calL(X_\calU)$ via
\begin{align*}
	T_\calU x_\calU = (\tilde T x)_\calU \quad \text{for all } x_\calU \in X_\calU \text{.}
\end{align*}

One reason for the usefulness of ultra products in operator theory is the fact that the induced operator $T_\calU$ shares a lot of spectral properties with the original operator $T$ and even improves some of them:

\begin{proposition} \label{prop_spectrum_for_operators_on_ultra_power}
	Let $T \in \calL(X)$ for a complex Banach space $X$ and let $\calU$ be a free ultra filter on $\bbN$.
	\begin{itemize}
		\item[(a)] For the spectra of $T_\calU$ and $T$ we have $\sigma(T_\calU) = \sigma(T)$. \par 
		\item[(b)] For the approximate point spectra $\sigma_{\operatorname{appr}}$ we have
			\begin{align*}
				\sigma_{\operatorname{pnt}}(T_\calU) = \sigma_{\operatorname{appr}}(T_\calU) = \sigma_{\operatorname{appr}}(T) \text{.}
			\end{align*}
	\end{itemize}
	\begin{proof}
		See \cite[Theorem 4.1.6]{Meyer-Nieberg1991} or \cite[Proposition V.1.3 and Theorem V.1.4]{Schaefer1974}.
	\end{proof}
\end{proposition}

If $X$ is a real or complex Banach space and $S,T \in \calL(X)$, then we have $||T_\calU|| = ||T||$ and $(ST)_\calU = S_\calU T_\calU$. This is very easy to see and will be used tacitly below. For our application of ultra powers in the next subsection, it is important to know how they behave in connection with complexifications of real Banach spaces. This is described by the following proposition.

\begin{proposition} \label{prop_ultra_powers_and_complexification}
	Let $X_\bbC$ be a complexification of a real Banach space $X$ and let $T_\bbC$ be the complex extension of an operator $T \in \calL(X)$. Let $\calU$ be a free ultra filter on $\bbN$. Then $(X_\bbC)_\calU$ is a complexification of $(X)_\calU$ via the embedding
	\begin{align*}
		(X)_\calU & \to (X_\bbC)_\calU \text{,} \\
		x + c_{0,\calU}(X) & \mapsto  x + c_{0,\calU}(X_\bbC)
	\end{align*}
	and the operator $(T_\bbC)_\calU$ is the complex extension of the operator $T_\calU$.
	\begin{proof}
		The proof is straightforward and therefore left to the reader.
	\end{proof}
\end{proposition}

The last result that we need on ultra powers of Banach spaces is the stability of $L^p$-spaces with respect to the construction of ultra powers:

\begin{proposition} \label{prop_ultra_powers_of_capital_l_p_spaces}
	Let $(\Omega, \Sigma, \mu)$ be a measure space, let $1 \le p < \infty$ and let $X$ be the real-valued function space $X := L^p(\Omega,\Sigma,\mu; \bbR)$. If $\calU$ is a free ultra filter on $\bbN$, then there is a measure space $(\tilde \Omega, \tilde \Sigma, \tilde \mu)$ such that the ultra power $X_\calU$ is isometrically isomorphic to $L^p(\tilde \Omega, \tilde \Sigma, \tilde \mu; \bbR)$.
	\begin{proof}
		See \cite[Theorem 3.3 (ii)]{Heinrich1980}.
	\end{proof}
\end{proposition}

\subsection{Contraction semigroups on $L^p$-spaces} \label{subsection_contraction_semigroups_on_capital_l_p_spaces}

Now, we will use the ultra power technique described above together with the ideas of Section \ref{section_peripheral_point_spectrum_on_projectively_non_hilbert_spaces} to analyse the spectrum of certain semigroups on $L^p$-spaces. To this end, consider a $C_0$-semigroup $(e^{tA})_{t\ge 0}$ on a (real or complex) Banach space and suppose that its \emph{growth bound} $\omega(A)$ is larger than $- \infty$; the semigroup $(e^{tA})_{t \ge 0}$ is called \emph{norm continuous at infinity} if it fulfils the condition
\begin{align*}
	\lim_{t \to \infty} \limsup_{h \to 0} ||e^{(t+h)(A - \omega(A))} - e^{t(A - \omega(A))}|| = 0 \text{.}
\end{align*}
Of course a $C_0$-semigroups on a real Banach space $X$ is norm continuous at infinity if and only if its complex extension to any complexification $X_\bbC$ of $X$ is so. The class of $C_0$-semigroups which are norm continuous at infinity contains the class of all $C_0$-semigroups which are eventually norm continuous and fulfil $\omega(A) > -\infty$. \par
The notion of norm continuity at infinity was introduced by Martinez and Mazon in \cite[Definition 1.1]{Martinez1996}, where they showed several spectral properties of those semigroups. We will need the following of those properties in the sequel:

\begin{proposition} \label{prop_spectral_properties_of_sg_norm_coninuous_at_infty}
	Let $(e^{tA})_{t \ge 0}$ be a $C_0$-semigroup on a complex Banach space. Suppose that $\omega(A) > - \infty$ and that $(e^{tA})_{t \ge 0}$ is norm-continuous at infinity.
	\begin{itemize}
		\item[(a)] Let $\Gamma_t := \{\lambda \in \bbC: |\lambda| = r(e^{tA})\}$. Then the following partial spectral mapping theorem holds for each $t \ge 0$:
			\begin{align*}
				\sigma(e^{tA}) \cap \Gamma_t = e^{t\sigma(A)} \cap \Gamma_t \text{.}
			\end{align*}
		\item[(b)] There is an $\varepsilon > 0$ such that the set
			\begin{align*}
				\{\lambda \in \sigma(A): \re \lambda \ge s(A) - \varepsilon\}
			\end{align*}
			is bounded.
	\end{itemize}
	\begin{proof}
		See \cite[Theorem 1.2]{Martinez1996} for (a) and \cite[Theorem 1.9]{Martinez1996} for (b).
	\end{proof}
\end{proposition}

The following theorem is the main result of this section.

\begin{theorem} \label{theorem_trivial_peripheral_spectrum_on_capital_l_p_spaces}
	Let $(\Omega, \Sigma, \mu)$ be a measure space and let $1 < p < \infty$, $p \not= 2$. Suppose that $(e^{tA})_{t \ge 0}$ is a $C_0$-semigroup on $X = L^p(\Omega, \Sigma, \mu; \bbR)$ which fulfils $\omega(A) > -\infty$, is norm continuous at infinity and uniformly asymptotically contractive. Then $\sigma(A) \cap i \bbR \subset \{0\}$.
\end{theorem}

Before we prove the theorem, let us briefly discuss its assumptions and its consequences: For some comments on the validity of Theorem \ref{theorem_trivial_peripheral_spectrum_on_capital_l_p_spaces} on other spaces, we refer to Remark \ref{remarks_main_thm_on_more_general_spaces} below. The assumption $\omega(A) > -\infty$ in the theorem is of course only a technical condition to ensure that the notion ``norm-continuous at infinity'' is well-defined. For eventually norm continuous semigroups, the following formulation of Theorem \ref{theorem_trivial_peripheral_spectrum_on_capital_l_p_spaces} might be more convenient:

\begin{corollary} \label{cor_trivial_peripheral_spectrum_on_capital_l_p_spaces}
	Let $(\Omega, \Sigma, \mu)$ be a measure space and let $1 < p < \infty$, $p \not= 2$. Suppose that $(e^{tA})_{t \ge 0}$ is a $C_0$-semigroup on $X = L^p(\Omega, \Sigma, \mu; \bbR)$ which is eventually norm continuous and uniformly asymptotically contractive. Then $\sigma(A) \cap i \bbR \subset \{0\}$.
	\begin{proof}
		If $\omega(A) = - \infty$, the assertion is trivial. If $\omega(A) > -\infty$, our semigroup is norm-continuous at infinity and thus the assertion follows from Theorem \ref{theorem_trivial_peripheral_spectrum_on_capital_l_p_spaces}.
	\end{proof}
\end{corollary}

For the asymptotic behaviour of our semigroups, we obtain the following corollary:

\begin{corollary} \label{cor_convergence_of_semigroups_on_l_p_spaces}
	Suppose that the assumptions of Theorem \ref{theorem_trivial_peripheral_spectrum_on_capital_l_p_spaces} or Corollary \ref{cor_trivial_peripheral_spectrum_on_capital_l_p_spaces} are fulfilled. Then $e^{tA}$ converges strongly as $t \to \infty$.
	\begin{proof}
		This follows from \cite[Corollary 2.6]{Arendt1988} or from \cite[Exercise V.2.25 (4) (ii)]{Engel2000}.
	\end{proof}
\end{corollary}

\begin{proof}[Proof of Theorem \ref{theorem_trivial_peripheral_spectrum_on_capital_l_p_spaces}]
	Except for the ultra-power technique involved, the major steps of proof are rather similar to those in the proof of Theorem \ref{theorem_point_spectrum_for_sg_on_projectively_non_hilbert_spaces}. \smallskip \par
	(a) We have $\omega(A) \le 0$ and for $\omega(A) < 0$ the assertion is trivial. So let $\omega(A) = 0$, and assume for a contradiction that the assertion of the theorem fails. Replacing $A$ by $c A$ for some $c > 0$, we may assume that $i,-i \in \sigma(A)$ and that $\sigma(A) \cap i\bbR \subset i \cdot [-1,1]$. Let $X_\bbC$ be a complexification of $X$ and let $(e^{tA_\bbC})_{t \ge 0}$ be the complex extension of the semigroup $(e^{tA})_{t \ge 0}$. We then have for each $t \ge 0$ that $\sigma(e^{itA_\bbC}) \cap \bbT = e^{it\sigma(A)} \cap \bbT \subset \{e^{i\varphi}: |\varphi| \le t\}$; the equality on the left follows from Proposition \ref{prop_spectral_properties_of_sg_norm_coninuous_at_infty} (a) since $r(e^{tA_\bbC}) = e^{t\omega(A_\bbC)} = 1$. \smallskip \par
	
	(b) Let $\calU$ be a free ultra filter on $\bbN$ and denote by $X_\calU$ and $(X_\bbC)_\calU$ the corresponding ultra powers of $X$ and $X_\bbC$. According to Proposition \ref{prop_ultra_powers_and_complexification}, $(X_\bbC)_\calU$ is a complexification of $X_\calU$ and the operator $(e^{tA_\bbC})_\calU$ is the complex extension of $(e^{tA})_\calU$. Hence, it follows from Proposition \ref{prop_spectrum_for_operators_on_ultra_power} that $\sigma_{\operatorname{pnt}}((e^{tA})_\calU) \cap \bbT = \sigma(e^{tA}) \cap \bbT$ for each $t \ge 0$, since the boundary of the spectrum of an operator is automatically contained in the approximate point spectrum. Besides that, we still have $\limsup_{t \to \infty} ||(e^{tA})_\calU|| \le 1$ However, note that the operator family $((e^{tA_\bbC})_\calU)_{t \ge 0}$ might not be strongly continuous on $(X_\bbC)_\calU$, although it of course still fulfils the semigroup law. \smallskip \par
	
	(c) We proceed to analyse some properties of several eigenspaces that will be used in the rest of the proof. Define
	\begin{align*}
		Y_1 := \operatorname{Eig}(e^{\frac{\pi}{2}i}, (e^{\frac{\pi}{2}A_\bbC})_\calU) \quad \text{and} \quad Y_2 = \operatorname{Eig}(e^{-\frac{\pi}{2}i}, (e^{\frac{\pi}{2}A_\bbC})_\calU) \text{.}
	\end{align*}
	By virtue of Lemma \ref{lem_eigenvectors_for_square_roots_of_eigenvalues} it follows that $\operatorname{Eig}(-1, (e^{\pi A_\bbC})_\calU) = Y_1 \oplus Y_2 =: Z_\bbC$. From the same Lemma and from the fact that $\sigma((e^{itA_\bbC})_{\calU}) \cap \bbT = \sigma(e^{itA_\bbC}) \cap \bbT \subset \{e^{i\varphi}: |\varphi| \le t\}$, we also conclude that $Y_1 = \operatorname{Eig}(e^{i\frac{\pi}{2^n}},(e^{\frac{\pi}{2^n}A_\bbC})_\calU)$ and $Y_2 = \operatorname{Eig}(e^{-i\frac{\pi}{2^n}}, (e^{\frac{\pi}{2^n}A_\bbC})_\calU)$ for all $n \in \bbN$. Thus,
	\begin{align}
		Y_1 \subset \operatorname{Eig}(e^{i\pi d}, (e^{\pi  d A_\bbC})_\calU) \quad \text{and} \quad Y_2 \subset \operatorname{Eig}(e^{-i\pi d}, (e^{\pi d A_\bbC})_\calU) \label{form_inclusions_for_eigenspaces_dyadic_numbers_lifted}
	\end{align}
	for each dyadic number $d \ge 0$, i.e.\ for each number of the form $d = \frac{k}{2^n}$, where $k,n \in \bbN_0$. Next, note that the mapping $[0,\infty) \to (X_\bbC)_\calU$, $t \mapsto (e^{tA_\bbC})_\calU x_\calU$ is continuous for each $x_\calU \in Y_1$ and each $x_\calU \in Y_2$: Indeed, let $x_\calU \in Y_1$ such that $||x_\calU|| = 1$, let $t \ge 0$ and $\varepsilon > 0$. For sufficiently large $n \in \bbN$, the norm continuity at infinity of $(e^{tA_\bbC})_{t \ge 0}$ implies that
	\begin{align*}
		& \limsup_{h \to 0} ||(e^{(\frac{n\pi}{2}+t+h)A_\bbC})_\calU - (e^{(\frac{n\pi}{2}+t)A_\bbC})_\calU|| = \\
		& = \limsup_{h \to 0} ||e^{(\frac{n\pi}{2}+t+h)A_\bbC} - e^{(\frac{n\pi}{2}+t)A_\bbC}|| \le \varepsilon \text{.} 
	\end{align*}
	Hence, we have
	\begin{align*}
		\varepsilon & \ge \limsup_{h \to 0} ||(e^{(\frac{n\pi}{2}+t+h)A_\bbC})_\calU x_\calU - (e^{(\frac{n\pi}{2}+t)A_\bbC})_\calU x_\calU|| = \\
		& = \limsup_{h \to 0} ||(e^{(t+h)A_\bbC})_\calU x_\calU - (e^{tA_\bbC})_\calU x_\calU|| \text{,}
	\end{align*}
	where the last equality follows from the fact that $(e^{\frac{\pi n}{2}A_\bbC})_\calU x_\calU = e^{i\frac{\pi n}{2}} x_\calU$. Since $\varepsilon > 0$ was an arbitrary number, we conclude that
	\begin{align*}
		\lim_{h \to 0} ||(e^{(t+h)A_\bbC})_\calU x_\calU - (e^{tA_\bbC})_\calU x_\calU|| = 0 \text{.}
	\end{align*}
	Similarly, we can show the continuity for $x_\calU \in Y_2$. The strong continuity of $t \mapsto (e^{tA_\bbC})_\calU$ on $Y_1$ and $Y_2$ together with (\ref{form_inclusions_for_eigenspaces_dyadic_numbers_lifted}) implies that
	\begin{align}
		\operatorname{Eig}(e^{it}, (e^{tA_\bbC})_\calU) \supset Y_1 \quad \text{and} \quad \operatorname{Eig}(e^{-it}, (e^{tA_\bbC})_\calU) \supset Y_2 \label{form_inclusions_for_eigenspaces_lifted}
	\end{align}
	for all $t \ge 0$. \smallskip \par
	
	(d) Let us now analyse the subspace $Z_\bbC = Y_1 \oplus Y_2$. This space coincides with the fixed space of the power-bounded operator $-(e^{\pi A_\bbC})_\calU$. Since $X_\calU$ is an $L^p$-space (see Proposition \ref{prop_ultra_powers_of_capital_l_p_spaces}) and $1 < p < \infty$, the space $X_\calU$ is reflexive and so is its complexification $(X_\bbC)_\calU$. Thus, $-(e^{\pi A_\bbC})_\calU$ is mean ergodic. The mean ergodic projection $P: (X_\bbC)_\calU \to (X_\bbC)_\calU$ has $Z_\bbC$ as its range. Moreover, $P_\bbC$ leaves $X_\calU$ invariant and the restriction $P_\bbC|_{X_\calU}$ is contractive, since $\limsup_{n \to \infty}||(-(e^{\pi A_\bbC})_\calU)^n|_{X_\calU}|| = \limsup_{n \to \infty} ||((e^{\pi A})_\calU)^n|| \le 1$. Note that $Z_\bbC$ is the complexification of the space $Z := P_\bbC(X_\calU) = X_\calU \cap Z_\bbC$ and that $Z$ is projectively non-Hilbert since it is the image of the projectively non-Hilbert space $X_\calU$ under the contractive projection $P_\bbC|_{X_\calU}$. \smallskip \par
	
	(e)	Finally, we want to apply Lemma \ref{lem_rotation_sg_with_one_eigenvalue_implies_contr_proj_hilbert_space} to the space $Z_\bbC = Y_1 \oplus Y_2$  and to the restricted semigroup $((e^{tA_\bbC})_\calU|_{Z_\bbC})_{t \ge 0}$. It follows from (\ref{form_inclusions_for_eigenspaces_lifted}), that the space $Z_\bbC$ is indeed invariant with respect to the operators $(e^{t A_\bbC})_\calU$, and the strong continuity assertion shown in (c) implies that the restricted semigroup $((e^{tA_\bbC})_\calU|_{Z_\bbC})_{t \ge 0}$ is indeed a $C_0$-semigoup. Due to (\ref{form_inclusions_for_eigenspaces_lifted}), this $C_0$-semigroup satisfies the spectral assumptions in Lemma \ref{lem_rotation_sg_with_one_eigenvalue_implies_contr_proj_hilbert_space}. Since it is also periodic (again due to (\ref{form_inclusions_for_eigenspaces_lifted})) and uniformly asymptotically contractive on $Z \subset X_\calU$, the semigroup acts in fact isometrically on $Z$. Thus, the conditions in Lemma \ref{lem_rotation_sg_with_one_eigenvalue_implies_contr_proj_hilbert_space} are fulfilled and hence, the lemma yields that $Z$ is \emph{not} projectively non-Hilbert. This is a contradiction to (d).
\end{proof}

After this proof, some remarks on the validity of Theorem \ref{theorem_trivial_peripheral_spectrum_on_capital_l_p_spaces} on other spaces are in order.

\begin{remarks} \label{remarks_main_thm_on_more_general_spaces} 
	(a) For contractive $C_0$-semigroups which are norm-continuous at infinity, the assertion of Theorem \ref{theorem_trivial_peripheral_spectrum_on_capital_l_p_spaces} remains true on $X := L^1(\Omega, \Sigma, \mu; \bbR)$. The proof of this assertion relies on the theory of Banach lattices and is therefore rather different from our proof above. The precise argument works as follows: \par 
	First note that $X_\bbC := L^1(\Omega,\Sigma, \mu; \bbC)$ is a complexification of our space $X$. Moreover, if the operator $e^{tA}$ is contractive, then its complex extension $e^{tA_\bbC}$ to $X_\bbC$ is contractive as well; this follows from \cite[Proposition 2.1.1]{Fendler1998}. Now, assume for a contradiction that $0\not= i \beta \in \sigma(A)$, where $\beta \in \bbR$. Then $e^{i\beta t} \in \sigma(e^{tA})$ for all $t \ge 0$. Moreover, it follows from Proposition \ref{prop_spectral_properties_of_sg_norm_coninuous_at_infty} that $\sigma(e^{tA}) \cap \bbT$ is contained in a sector of small angle in the right half plane if we choose $t > 0$ sufficiently small. However, since $e^{it\beta} \in \sigma(e^{tA})$, it follows from \cite[Corollary 2 to Theorem V.7.5]{Schaefer1974}	that $e^{i(2n+1)t \beta} \in \sigma(e^{tA})$ for each $n \in \bbZ$. This is a contradiction. \par 
	The result from \cite[Corollary 2 to Theorem V.7.5]{Schaefer1974} that we used in the last step is based on the spectral theory of positive operators on Banach lattices (although the operator under consideration is itself not positive) and on the special structure of $L^1$-spaces. Besides that, we note that the above argument strongly depends on the fact that $e^{tA}$ is contractive for \emph{small} $t$. Currently, the author does not know whether the assertion $\sigma(A) \cap i \bbR \subset \{0\}$ remains true on $L^1$-spaces if we consider uniformly asymptotically contractive instead of contractive semigroups. \par 
	(b) In Theorem \ref{theorem_trivial_peripheral_spectrum_on_capital_l_p_spaces} we do not really need to consider $L^p$-spaces. Instead, we can replace the $L^p$-space in the theorem by a real Banach space $X$ which fulfils that some ultra power $X_\calU$ of $X$ is still reflexive and projectively non-Hilbert. The theorem remains true for those spaces, since its proof only uses that $X_\calU$ is reflexive and projectively non-Hilbert. However, currently the author does not know any example of a real Banach space $X$ which is not an $L^p$-space, but has a reflexive and projectively non-Hilbert ultra power $X_\calU$.
\end{remarks}

One might wonder why we required the semigroup to be uniformly asymptotically contractive in Theorem \ref{theorem_trivial_peripheral_spectrum_on_capital_l_p_spaces}, which is slightly stronger than the weak asymptotic contractivity that we required in the Theorems \ref{theorem_point_spectrum_for_sg_on_extremely_non_hilbert_spaces} and \ref{theorem_point_spectrum_for_sg_on_projectively_non_hilbert_spaces}. The problem with weak (and even with strong) asymptotic contractivity is that it does not carry over to the lifted operators on an ultra power. The uniform condition $\limsup_{t \to \infty} ||e^{tA}|| \le 1$ however still holds for the lifted operators, since their norm coincides with the norm of the original operators. In fact, the assertion of Theorem \ref{theorem_trivial_peripheral_spectrum_on_capital_l_p_spaces} may fail for strongly (and in particular for weakly) asymptotically contractive semigroups:

\begin{example} \label{example_pointwise_asymp_contr_is_not_sufficient_for_triviality_of_periph_spec_sg}
	Let $1 < p < \infty$, $p \not= 2$, and endow $\bbR^2$ with the $p$-Norm. For each $n \in \bbN$, let $A_n$ be the operator on $\bbR^2$ whose representation matrix with respect to the canonical basis is given by
	\begin{align*}
		\begin{pmatrix}
			-\frac{1}{n} & -1 \\
			1 & - \frac{1}{n}
		\end{pmatrix} \text{.}
	\end{align*}
	We have $\sigma(A_n) = \{i - \frac{1}{n}, -i - \frac{1}{n}\}$, and the representation matrix of $e^{tA_n}$ with respect to the canonical basis is given by
	\begin{align*}
		e^{-\frac{1}{n}t}
		\begin{pmatrix}
			\cos t & - \sin t \\
			\sin t & \cos t
		\end{pmatrix}
	\end{align*}
	We clearly have $||e^{tA_n}|| \to 0$ as $t \to \infty$. Now, consider the vector-valued $l^p$-space $X = l^p(\bbN;\bbR^2)$ (which is isometrically isomorphic to the space $l^p(\bbN;\bbR)$) and let $A := \bigoplus_{n=1}^\infty A_n \in \calL(X)$. Then $A$ generates a $C_0$-semigroup $(e^{tA})_{t \ge 0}$ on $X$ and it is easy to see that this semigroup fulfils $\limsup_{t \to \infty} ||e^{tA}x|| \le ||x||$ (and in fact even $e^{tA}x \to 0$ as $t \to \infty$) for each $x \in X$, but that $\limsup_{t \to \infty} ||e^{tA}|| > 1$. Thus, the semigroup is strongly, but not uniformly asymptotically contractive. Furthermore, we have $\sigma(A) \cap i \bbR = \{i,-i\} \not \subset \{0\}$, so the assertion of Theorem \ref{theorem_trivial_peripheral_spectrum_on_capital_l_p_spaces} does not hold for the semigroup $(e^{tA})_{t \ge 0}$. However, note that $i$ and $-i$ are not eigenvalues of $A$, which is in accordance with Theorem \ref{theorem_point_spectrum_for_sg_on_projectively_non_hilbert_spaces}.
\end{example}

Example \ref{example_pointwise_asymp_contr_is_not_sufficient_for_triviality_of_periph_spec_sg} is based on an idea that was used in \cite{ArendtInPrep} to construct a counter example in the theory of eventually positive semigroups. Note that we can easily modify Example \ref{example_pointwise_asymp_contr_is_not_sufficient_for_triviality_of_periph_spec_sg} to obtain a semigroup which is strongly asymptotically contractive, but even fulfils $\sigma(A) \cap i\bbR = i \cdot [-1,1]$: We simply have to multiply the off-diagonal entries of the representation matrices of $A_n$ by numbers $\alpha_n$, where $(\alpha_n)_{n \in \bbN} \subset [0,1]$ is a sequence which is dense in $[0,1]$.

\subsection{Contractive single operators on $L^p$-spaces}

Using the ultra power technique described at the beginning of this section, we can also derive a result on single operators which is similar to Theorem \ref{theorem_point_spectrum_of_op_on_projectively_non_hilbert_space}, but now for the spectrum instead of the point spectrum. Actually the proof for the single operator case is much easier than for the semigroup case, since we can simply lift a single operator to an ultra power of $X$ and apply Theorem \ref{theorem_point_spectrum_of_op_on_projectively_non_hilbert_space}.

\begin{theorem} \label{theorem_spectrum_of_op_on_projectively_non_hilbert_space}
	Let $(\Omega, \Sigma, \mu)$ be a measure space, let $1 < p < \infty$, $p \not= 2$ and set $X := L^p(\Omega, \Sigma, \mu; \bbR)$. Suppose that $T \in \calL(X)$ is uniformly asymptotically contractive. If $\sigma(T) \cap \bbT$ is finite, then it only consists of roots of unity.
	\begin{proof}
		Let $\calU$ be a free ultra filter on $\bbN$; the ultra power $X_\calU$ of $X$ is an $L^p$-space again and we have $\sigma(T) \cap \bbT = \sigma_{\operatorname{pnt}}((T)_\calU) \cap \bbT$. Since the operator $T_\calU$ is defined on a reflexive $L^p$-space with $p\not=2$ and fulfils the condition $\limsup_{n \to \infty} ||T_\calU^n|| \le 1$, we can apply Theorem \ref{theorem_point_spectrum_of_op_on_projectively_non_hilbert_space} to conclude that $\sigma_{\operatorname{pnt}}(T_\calU) \cap \bbT$ consists only of roots of unity.
	\end{proof}
\end{theorem}

Comparing our results on $C_0$-semigroups with our results on single operators, it seems that the semigroup results are much more satisfying. For example, the boundedness condition on $\sigma_{\operatorname{pnt}}(A) \cap i\bbR$ from Theorem \ref{theorem_point_spectrum_for_sg_on_projectively_non_hilbert_spaces} and the norm continuity assumption from Theorem \ref{theorem_trivial_peripheral_spectrum_on_capital_l_p_spaces} will be satisfied on many occasions, even without any compactness assumptions. By contrast, the condition from Theorems \ref{theorem_spectrum_of_op_on_projectively_non_hilbert_space} and \ref{theorem_point_spectrum_of_op_on_projectively_non_hilbert_space} that the peripheral (point) spectrum $\sigma(T) \cap \bbT$ (respectively $\sigma_{\operatorname{\operatorname{pnt}}}(T) \cap \bbT$) be finite seems to be rather strong. It would be interesting to know whether the same conclusions for single operators hold under weaker a priori assumptions on the spectrum.

\subsection*{Acknowledgements} I wish to thank Stephan Fackler for bringing the result on contractive projections in $L^p$-spaces in \cite[Theorem~6]{Tzafriri1969} to my attention. During his work on this article, the author was supported by a scholarship of the ``Lan\-des\-gra\-duier\-ten\-f\"or\-der\-ung Baden-W\"urttemberg''.

\bibliographystyle{plain}

\end{document}